\documentclass[12pt,reqno]{amsart}
\usepackage{amsbsy,amsfonts,amsmath,amssymb,amscd,amsthm}
\usepackage{mathrsfs,float}
\usepackage[mathcal]{euscript}
\usepackage{subfigure,enumitem,graphicx,epstopdf}
\usepackage[a4paper,text={6.1in,8in},centering,marginparwidth=23mm]{geometry}
\usepackage{pxfonts,epstopdf}
\usepackage[colorlinks=true,linkcolor=blue,citecolor=green]{hyperref}
\usepackage{srcltx}
\usepackage[margin=20pt,font=small,labelfont=bf,textfont=it]{caption}
\usepackage{comment}
\usepackage{mathtools}

\theoremstyle{plain}

\newtheorem{thm}{Theorem}[section]
\newtheorem{cor}[thm]{Corollary}
\newtheorem{lem}[thm]{Lemma}

\newtheorem{prop}[thm]{Proposition}

\theoremstyle{definition}

\newtheorem{rem}[thm]{Remark}
\newtheorem{question}[thm]{Question}

\newcommand{\Ee}{\mathscr{E}}

\newcommand{\T}{\mathbb{T}}

\newcommand{\Or}{\mathcal{O}}

\newcommand{\U}{\mathcal{U}}
\newcommand{\V}{\mathcal{V}}

\newcommand{\eqdef}{\stackrel{\scriptscriptstyle\rm def}{=}}

\small\normalsize
\let\oldmarginpar\marginpar
\renewcommand\marginpar[1]{\-\oldmarginpar[\raggedleft\tiny #1]%
{\raggedright\tiny #1}}

\begin{document}

\title[Entropy and dimension of irregular sets]{Topological entropy and Hausdorff dimension of irregular sets for non-hyperbolic dynamical systems}

\author[Barrientos et al.]{Pablo G. Barrientos}
\address[Barrientos]{Instituto  de Matem\'atica e Estat\'istica, Universidade Federal Fluminense, Gragoata Campus, Rua Prof.\ Marcos Waldemar de Freitas Reis, S/n-Sao Domingos, Niteroi - RJ, 24210-201, Brazil}
\email{pgbarrientos@id.uff.br}

\author[]{Yushi Nakano}
\address[Nakano]{Department of Mathematics, Tokai University, 4-1-1 Kitakaname, Hiratuka, Kanagawa, 259-1292, Japan}
\email{yushi.nakano@tsc.u-tokai.ac.jp}

\author[]{Artem Raibekas}
\address[Raibekas]{
Instituto  de Matem\'atica e Estat\'istica, Universidade Federal Fluminense, Gragoata Campus, Rua Prof.\ Marcos Waldemar de Freitas Reis, S/n-Sao Domingos, Niteroi - RJ, 24210-201, Brazil}
\email{artemr@id.uff.br}

\author[]{Mario Roldan}
\address[Roldan]{Departamento de Matem\'atica, Universidade Federal de Santa Catarina
Campus Universit\'ario Trindade,
Florian\'opolis - SC, 88.040-900 Brazil}
\email{roldan@impa.br}

\subjclass[2010]{Primary: 37B40, 37C40, 37C45.}
\keywords{Irregular set,  topological entropy, Hausdorff dimension, non-hyperbolic systems}

\date{}

\begin{abstract}
We systematically investigate examples of non-hyperbolic dynamical systems having irregular sets of full topological entropy and full Hausdorff dimension.
The examples include some partially hyperbolic systems and geometric Lorenz flows.
We also pose interesting questions for future research.
\end{abstract}

~\vspace{1.5cm}

\maketitle
\section{Introduction}
Let us consider a continuous map $f$ of a totally bounded metric space $X$.
We define the irregular set $I(f)$ of  $f$ as the set  of  points $x \in X$ such that there exists a continuous function $\phi: X\to \mathbb{R}$ for which the time average
\[
    \lim_{n\to \infty} \frac{1}{n}\sum_{i=0}^{n-1}\phi(f^i(x))
\]
does not exist. It turns out that $I(f)$ has zero measure for any
invariant measure by Birkhoff's ergodic theorem. Although this set
is negligible from the point of view of Ergodic Theory, it is
still possible to be topologically large or to have large
topological entropy and Hausdorff dimension. In fact, dynamical
systems with  irregular sets of full topological entropy and full
Hausdorff dimension  have been intensively studied in the
literature~\cite{PP84,BS00,FFW01,EKL05,T10,T17,ma2017hausdorff,DOT,barreira2018topological,FKOT19}.
However, most of these works only deal with continuous/smooth maps
having inherent properties from (non)uniform hyperbolicity, such
as  the \emph{specification-like property} or the \emph{shadowing
property}. In this paper, we establish full topological entropy
and full Hausdorff dimension of the irregular set for abundant
examples of dynamics without such properties, but well
approximated by some hyperbolic structures. To our best knowledge,
this is the first systematic investigation of such phenomena
beyond the specification-like or shadowing property.

{In Section \ref{sec2}, we study the topological entropy of
irregular sets. In lower dimensions, we show that $I(f)$ has full
topological entropy when $f$ is a continuous map
or a piecewise monotonic map of a compact
interval (Theorem~\ref{thm-entropy-one-dimensional}) and when $f$
is  a $C^{1+\alpha}$-diffeomorphism of a compact surface
(Corollary~\ref{thm-entropy-two-dimensional}). We obtain the same
result for some interesting higher dimensional
systems. These include partially hyperbolic diffeomorphisms
%in 3-dimensional nilmanifolds  other than the torus, partially hyperbolic diffeomorphisms
having a one-dimensional center direction and minimal strong
foliations, partially hyperbolic diffeomorphisms on three dimensional
nilmanifolds,  perturbations of time-one map of Anosov flows,
generic symplectic diffeomorphisms with no dominated splitting and
diffeomorphisms isotopic to Anosov (see~\S\ref{Examples in higher
dimensions}).

%{
In Section \ref{sec3}, we study the Hausdorff dimension of  irregular sets.
We review the known results in lower dimensions on surfaces and
in parametric families, concluding that $I(f)$ has full Hausdorff dimension
in these examples. In Section~\ref{ss:0219b} we look at skew-products over the
symbolic shift with interval fibers. These examples are related to
partially hyperbolic sets, such as blenders and porcupines.
Afterward, we study mostly contracting partially hyperbolic systems
(\S\ref{mostlycontracting}) and systems having no dominated splitting
(\S\ref{nodomination}), where we conclude full Hausdorff dimension of
the irregular set in the generic context.
%}

In Section \ref{s:Lorenz}, we prove that the irregular set of
geometric Lorenz flows $C^1$-robustly has full topological entropy
and full Hausdorff dimension  (Theorem~\ref{prop:0904}). The
longer and more technical proofs  are collected   in
Section~\ref{s:proof} in order not to obscure the main ideas of
this paper.

\section{Systems with the irregular set having full topological entropy}
\label{sec2}

Since the irregular set  $I(f)$ is not necessarily compact,  we need to adapt the notion of
  topological entropy   to this setting.
In what follows, as in the previous literature \cite{BS00,T10}, we consider the topological entropy $h_{\mathrm{top}}(A,f)$ in the sense of Pesin and Pitskel'~\cite{PP84}.
%for each map $f$ of a totally bounded metric space $X$ and an  invariant subset $A$ of $X$ on which $f$ is continuous.
This entropy is defined for a continuous map $f$ of a totally bounded metric space $X$ and an  invariant subset $A$ as a \emph{Carath\'eodory characteristic of dimension type} and we refer to~\cite{Pesin98} for the precise definition.
Observe that this notion  allows us to deal also with a discontinuous map $f: X\to X$ by restricting $f$  to $Y=X\setminus \{f^{-n}(Z): n\geq 0\}$  where $Z$ is the set of discontinuity points.
The variational principle also holds for this topological entropy
and, in the case that $f$ is continuous, $h_{\mathrm{top}}(A,f)$ coincides with the \emph{Bowen-Hausdorff topological entropy} of a (not necessarily compact) set $A$  introduced by Bowen in~\cite{Bow73}.
In particular, when $A$ is a compact subset of $X$,  $h_{\mathrm{top}}(A,f)$ coincides with the usual topological entropy (i.e.~the Bowen-Dinaburg topological entropy \cite{Dinaburg70,bowen1971entropy}).
We simply denote $h_{\mathrm{top}}(X,f)$ by $h_{\mathrm{top}}(f)$.

\subsection{Approximation of entropy by horseshoes}\label{S:2.1}

We start by the following basic observation, which involves
invariant compact sets with the specification property
(see definition
in~\cite[Def.~3]{kwietniak2016panorama}).

\begin{prop} \label{prop:1}
Let $f$ be a continuous map of a totally bounded metric space $X$.
Then,
$$
   h_{\mathrm{top}}(I(f),f) \geq \sup  h_{\mathrm{top}}(\Lambda,f)
$$
where the supremum is with respect to all compact subsets $\Lambda$ of $X$, which are invariant and satisfy the specification property for some iterate of $f$.
\end{prop}
\begin{proof}
Assume that $\Lambda$ is a compact $f^k$-invariant subset of $X$ with the specification property for some $k\geq 1$.
It is easy to see that $I(f^k) \subset I(f)$\footnote{{This can be seen by using the following decomposition which relates the finite Birkhoff sum of $f$ with that of $f^k$ and then analyzing the limits inferior and superior,
$$
\frac{1}{n} \sum_{i=0}^{n-1} \phi(f^i(x)) = \frac{\ell_n}{n} \cdot \frac{1}{\ell_n} \sum_{j=0}^{\ell_n-1} \phi(f^{jk}(x)) + \frac{1}{n} \sum_{r=1}^{k-1} \sum_{j=0}^{\ell_n-1} \phi(f^{jk+r}(x)) + \frac{1}{n}\sum_{r=0}^{r_n} \phi(f^{\ell_n k +r}(x))
$$
where $0\leq r_n <k$ and $\ell_nk+r_n=n-1$.
}}.
Also, observe that in general, one has that $I(f^k)\cap \Lambda \subset I(f^k|_{\Lambda})$.
On the left-hand of the inclusion, we consider the time average of real-valued continuous observables $\phi$ on $X$, while on the right-hand, observables are only defined on $\Lambda$.
However, since $\Lambda$ is a closed set of $X$, by Tietze extension theorem we can extend continuously the observables to the whole space $X$ and get the other inclusion.
Thus, we have that $I(f^k|_{\Lambda})=I(f^k)\cap \Lambda$.
Hence, according to~\cite[Corollary~3.11]{EKL05} it follows that
\begin{equation} \label{eq:z}
h_{\mathrm{top}}(\Lambda,f^k) = h_{\mathrm{top}}(I(f^k|_{\Lambda}),f^k) =
h_{\mathrm{top}}(I(f^k)\cap \Lambda,f^k) \leq h_{\mathrm{top}}(I(f)\cap
\Lambda,f^k).
\end{equation}
On the other hand, from \cite[Proposition~2(d)]{Bow73}, we get
\begin{equation} \label{eq:zz}
h_{\mathrm{top}}(I(f)\cap \Lambda,f^k)=k\cdot h_{\mathrm{top}}(I(f)\cap \Lambda,f)
\quad  \text{and}  \quad h_{\mathrm{top}}(\Lambda,f^k)= k\cdot
h_{\mathrm{top}}(\Lambda,f).
\end{equation}
Putting together~\eqref{eq:z} and~\eqref{eq:zz} we obtain that
$$
    h_{\mathrm{top}}(\Lambda,f) \leq   h_{\mathrm{top}}(I(f)\cap \Lambda,f)\leq h_{\mathrm{top}}(I(f),f)
$$
which completes the proof.
\end{proof}

We shall apply the above basic fact to a diffeomorphism $f$ of a
compact manifold $M$ as follows. Recall that an ergodic
$f$-invariant probability measure $\mu$ is said to be
\emph{hyperbolic} if  all of its Lyapunov exponents  are nonzero.
On the other hand,  an \emph{elementary horseshoe} of $f$ is a
Cantor set $\Lambda$ which is a topologically mixing locally
maximal invariant hyperbolic set for some iterate of $f$. Let $f$
be a $C^{1+\alpha}$-diffeomorphism (i.e.~a $C^1$-diffeomorphism
with $\alpha$-H\"older derivative)
%or a $C^1$-diffeomorphism with
%a dominated splitting\footnote{A splitting $T_xM =E(x)\oplus F(x)$
%is called dominated if it is invariant under the derivative,
%varies continuously with $x$, and there exists $C>0$, $\lambda< 1$
%such that for every $x\in M$ and every pair of unit vectors $u\in
%E(x), v\in F(x)$, one has $\| Df^n(x)u \|\cdot\| Df^n(x)v
%\|^{-1}<C\lambda^n$, for every $n\geq 1$
%(cf.~\cite{bonatti2006dynamics}).},
and consider  a hyperbolic ergodic $f$-invariant probability
measure $\mu$. Then, according to \cite[Theorem 1]{G16} and
\cite{KH95}, there is a sequence of elementary horseshoes
$(\Lambda_n)_{n\geq 1}$ of $f$ such that
$h_{\mathrm{top}}(\Lambda_n,f) \to h_{\mu}(f)$ as $n\to\infty$
where $h_\mu(f)$ denotes the (Kolmogorov-Sinai) metric entropy of
$\mu$. Since $\Lambda_n$ is an elementary horseshoe (and in
particular expansive) for some iterate $f^k$ of $f$,
%a topologically mixing locally maximal invariant hyperbolic compact set (in particular expansive)
then the specification property holds on $\Lambda_n$ (for $f^k$)~\cite{denker}.
Hence, from Proposition~\ref{prop:1} it follows that
\begin{equation} \label{eq:1}
    h_{\mathrm{top}}(I(f),f)\geq h_\mu(f).
\end{equation}
%To conclude, one gets the following corollary.
In light of the above observation, we  introduce the  \emph{hyperbolic entropy} $H(f)$ of $f$ as
\[
H(f) =\sup \big\{ h_\mu(f) :   \text{$\mu$ is a  hyperbolic
ergodic $f$-invariant probability measure}\big\}
\]
and get the following.

\begin{cor} \label{cor:1}
Let $f$ be a $C^{1+\alpha}$-diffeomorphism of a compact manifold
$M$. Then, the topological entropy of the irregular set  of $f$ is
bounded from below by the hyperbolic entropy, that~is,
\[
h_{\mathrm{top}}(I(f),f) \geq  H(f).
\]
\end{cor}
%\begin{proof} By an straitforward 'diagonal' argument, from Katok's horseshoes
%approximation theorem we can find a sequence of elementary
%horseshoes $\Lambda_n$ so that $h_{\mathrm{top}}(\Lambda_n,f)$ approaches
%$H(f)$. Since $\Lambda_n$ is a topologically mixing locally
%maximal invariant hyperbolic compact set (in particular expansive)
%for some iterated $f^k$ of $f$ then satisfies the specification
%property (for $f^k$)~\cite{Bow}. In particular, from
%Proposition~\ref{prop:1} follows that $h_{\mathrm{top}}(I(f),f)\geq H(f)$.
%\end{proof}
%In light of the previous corollary, we formally define the key quantity in this paper.
%\begin{dfn}
%sa
%\end{dfn}
{We observe that the above result already appeared in~\cite[Corollary B]{DT17}, nevertheless we decided to include the previous proof for the completeness of the article}.

{
\begin{rem} \label{rem_dominated}
Corollary~\ref{cor:1} can be similarly stated for $C^1$-diffeomorphisms having a dominated splitting by using a version of the theorem of Katok on approximation by horseshoes obtained in~\cite{G16} (see also \cite{ST10}).
%(see also~\cite{ST15}).
To do this, it requires assuming that
the hyperbolic measures have an Oseledets splitting coinciding with the dominated splitting in the hypothesis.
 \end{rem}
}

Thus, one approach to obtain full topological entropy of the irregular set $I(f)$ is to find hyperbolic measures whose metric entropy approaches the topological entropy
%of $f$ on
of the whole space $X$.
Similar estimates  can also be obtained for the topological pressure of  irregular sets  using the arguments and the results in~\cite{T10} and~\cite{sanchez2017approximation}.

\subsection{Examples in dimension one}
Recently, the second author of this paper obtained in~\cite{NY20} the following result.
The irregular set of a piecewise monotonic map of a compact interval  has full topological entropy if $f$ is transitive and the set of periodic measures of $f$ is dense in the set of ergodic measures of $f$.
Recall that a map $f:I\to I$ is said to be a \emph{piecewise monotonic map} of a compact interval $I$ if there exists a partition of $I$ into finitely many pairwise disjoints intervals
$I_1,\dots,I_k$ such that $f|_{I_j}$ is strictly monotone and continuous for $j=1,\dots,k$.
Here, we extend this result to \emph{any} continuous map and \emph{any} piecewise monotonic map of a compact interval.

%\begin{itemize}
%\item some words on the notion of topological entropy introduces by Misiurewicz
%for PMM.
%\item some words on [YN08].
%\item definition of strict $p$-horseshoe for interval maps.
%\end{itemize}

\begin{thm} \label{thm-entropy-one-dimensional}
Let $f$ be either a continuous or piecewise monotonic map of a compact interval.
Then, the irregular set has full topological entropy, that is,
$h_{\mathrm{top}}(I(f),f)=h_{\mathrm{top}}(f).$
\end{thm}

The proof of Theorem \ref{thm-entropy-one-dimensional} will be given in Section \ref{S:5.1}.
%To prove Theorem \ref{thm-entropy-one-dimensional},
In the proof, we will  use the previously mentioned idea of approximation of the topological entropy of $f$ by horseshoes, which is now achieved   by Misiurewicz's theorem.
We remark that these horseshoes are of $f$ and not of a lift of $f$ (that is, the Hofbauer's  Markov diagram of $f$) as in the approach of~\cite{NY20}.

%{Put proof and rest of comments in the end of article}
% {to do: 1) one dimensional Lorenz-like maps,
%approximation in entropy by Misurewicz
%horseshoes, approximation in Hausdorff dimension by Pacifico.
%2) Look for references for the quadratic map }

\subsection{Examples in dimension two}
%On the other hand, if  the topological entropy of $f$ is zero then we have nothing to do to show full topological entropy of the irregular set (i.e.~$h_{\mathrm{top}}(I(f),f) =h_{\mathrm{top}}(f) $), so
%%throughout this subsection we tacitly
%we assume that the topological entropy of $f$ is positive.
By Ruelle's inequality in dimension two, any ergodic invariant
probability measure of a $C^1$-diffeomorphism with positive metric
entropy is hyperbolic. Thus, by the variational principle the
topological entropy of any surface $C^1$-diffeomorphism $f$
coincides with the hyperbolic entropy of $f$. Hence,
Corollary~\ref{cor:1} implies that the irregular set of $f$ has
full topological entropy if $f$ is a
$C^{1+\alpha}$-diffeomorphism.
%or a $C^1$-diffeomorphism with a dominated splitting.

Moreover, one can restrict the attention to any  homoclinic class
$H(\Or)$ with positive topological entropy. Recall that a
homoclinic class of a hyperbolic periodic orbit $\Or$ is the
closure of the transverse intersections between the stable and
unstable manifold of $\Or$. Any homoclinic class is a compact
transitive invariant set. As a consequence again of the Ruelle
inequality and the variational principle for the restriction of
$f$ to $H(\Or)$ we obtain that $h_{\mathrm{top}}(H(\Or),f)$ can be
approximated by the metric entropy $h_\mu(f)$ of a hyperbolic
invariant measure $\mu$ supported on $H(\Or)$.
Thus,  from
 %the Katok horseshoe approximation theorem (see~\eqref{eq:1} and
 Corollary~\ref{cor:1} and the fact that $I(f\vert _{H(\Or)} ) =I(f) \cap H(\Or)$
%(refer to the proof of Proposition~\ref{prop:1}),
it follows that $I(f)$ has full topological entropy in any homoclinic class $H(\Or)$,
that is,
$$h_{\mathrm{top}}(I(f)\cap H(\Or),f)=h_{\mathrm{top}}(H(\Or),f).$$
%\marginpar{\tiny dominated splittingã?®å ´å?ˆã?¯ï¼Ÿ}
%
To summarize the above arguments, we get  the following:
\begin{cor} \label{thm-entropy-two-dimensional}
Let $f$ be a $C^{1+\alpha}$-diffeomorphism
%(or a $C^1$-diffeomorphism with dominated splitting)
of a compact surface. Then the  irregular set $I(f)$ has full
topological entropy. Moreover, $I(f)$ has full topological entropy
in any homoclinic class $H(\Or)$.
\end{cor}

We were recently informed that the above result was also obtained
by D.~Sanhueza in his Ph.D.~thesis~\cite{S20}. See also~\cite[Corollary~C]{DT17}.

%\subsubsection{Smooth case of homoclinic classes} For $C^\infty$
%maps $f$ from a compact manifold to itself Newhouse proved in
%\cite{New89} that $\mu\mapsto h_{\mu}(f)$ from the set of
%$f$-invariant probability measures to $\mathbb{R}$ is
%uppersemicontinuous. %Since uppersemicontinuous maps defined on a
%%compact manifold attains its supremum,
%It follows as a consequence that any $C^{\infty}$ diffeomorphism
%of a compact manifold has an ergodic measure of maximal entropy
%(in general, this fails in any finite smoothness). Now, we
%consider the special case of compact surfaces. Suppose that $f$ is
%a $C^{\infty}$ surface diffeomorphism with a homoclinic class
%$H(\Or)$ such that $h_{\mathrm{top}}(H(\Or),f)>0$. Recall that a homoclinic
%class of a hyperbolic periodic orbit $\Or$ is the closure of the
%transverse intersection between the stable and unstable manifold
%of $\Or$. Then, as consequence again of the uppersemicontinuity,
%there exists a $f$-invariant probability measure $\mu$ supported on $H(\Or)$
%such that $h_{\mathrm{top}}(H(\Or),f)=h_{\mu}(f)$. As we noted previously,
%this measure $\mu$ is necessarily a hyperbolic measure. Thus, by
%Katok horseshoe approximation result (see~\eqref{eq:1} and
%Corollary~\ref{cor:1}) implies the following result:
%
%\begin{cor} Let $f$ be
%a $C^{\infty}$ diffeomorphism of a compact surface. Then set of
%irregular points has full topological entropy in any homoclinic
%class $H(\Or)$, that is, $$h_{\mathrm{top}}(I(f)\cap
%H(\Or),f)=h_{\mathrm{top}}(H(\Or),f).$$
%\end{cor}

\subsection{Examples in higher dimensions}
\label{Examples in higher dimensions}

In what follows, we describe some interesting examples in dimension greater than two whose irregular sets have full topological entropy.

\subsubsection{Non-hyperbolic homoclinic classes} \label{sec:foliation}
First, we focus on the case of non-hyperbolic homoclinic classes
with index-variation, {a question that was asked}
in~\cite[Question~1]{BKNRS20}. Let $M$ be a compact manifold with
$\dim (M) \geq 3$ and denote by $\mathrm{Diff}^1(M)$ the set of
$C^1$-diffeomorphisms of $M$. Following the recent
works~\cite{DGS19,YZ19} we  consider the open set $\U \subset
\mathrm{Diff}^1(M)$  consisting of partially hyperbolic
systems\footnote{A diffeomorphism $f$ is called partially
hyperbolic if there exists a decomposition $TM=E^s\oplus E^c\oplus
E^u$ of the tangent bundle into three continuous invariant
sub-bundles such that (i) $Df\lvert E^s$ is uniformly contracting,
$Df\lvert E^u$ is uniformly expanding and (ii) for any unit
vectors $v^s\in E^s$, $v^c\in E^c$ , $v^u\in E^u$ and any $x\in
M$, we have that $\| Df(x)v^s \|\cdot \| Df(x)v^c \|^{-1}<1/2$,
and $\| Df(x)v^c \|\cdot \| Df(x)v^u \|^{-1}<1/2$.} a with
one-dimensional central bundle, having minimal strong stable and
unstable foliations and a pair of hyperbolic periodic points $P$
and $Q$ with different indices (dimension of the stable bundle).
In particular, these systems are $C^1$-robustly transitive and
non-hyperbolic. Moreover, there exists an open and dense subset
$\V$ of $\U$ such that for any $f\in\V$, it holds that
$H(P)=H(Q)=M$.
%(recall that $H(\Or(P))$ is the homoclinic class of
%where  $\Or(x)$ is the  orbit of $f$ issued from $x\in M$.
%and $f$ has a $cs$-blender and a $cu$-blender (see \cite{DGS19,YZ19} for their definitions).
This kind of diffeomorphisms is quite
abundant among partially hyperbolic skew-products in dimension 3
having circles as the central fibers. The authors in~\cite{DGS19,YZ19} proved
that for all $f\in\V$ the hyperbolic entropy  of $f$ is equal to
the topological entropy. Consequently, it follows from
{Corollary~\ref{cor:1}, Remark~\ref{rem_dominated}} that the irregular set $I(f)$ has full
topological entropy for any $f\in\mathcal{V}$.

\subsubsection{Volume-preserving or symplectic diffeomorphisms with no dominated splitting} \label{sec:simpletic}
Let $\omega$ be a volume or a symplectic form of a compact
manifold $M$. We denote by $\Ee_{\omega}(M )$ the interior of the
set of $C^1$ diffeomorphisms which preserve $\omega$ and  having
no dominated splitting\footnote{A splitting $T_xM =E(x)\oplus
F(x)$ is called dominated if it is invariant under the derivative,
varies continuously with $x$, and there exists $C>0$, $\lambda< 1$
such that for every $x\in M$ and every pair of unit vectors $u\in
E(x), v\in F(x)$, one has $\| Df^n(x)u \|\cdot\| Df^n(x)v
\|^{-1}<C\lambda^n$, for every $n\geq 1$
(cf.~\cite{bonatti2006dynamics}).}. Buzzi, Crovisier and Fisher in
\cite{BCF18} proved that the topological entropy of a generic
$f\in\Ee_{\omega}(M)$ is the supremum of the topological entropy
of horseshoes of $f$. Although a priori these horseshoes $\Lambda$
need not be elementary, one can use the Spectral Decomposition
Theorem~\cite{KH95} and obtain an elementary subhorseshoe
$\Lambda'$ having the entropy of $\Lambda$. Thus, we can always
approximate $h_{\mathrm{top}}(f)$ by the topological entropy of
elementary horseshoes of $f$. Therefore, Proposition~\ref{prop:1}
concludes that  the irregular set of generic diffeomorphisms $f\in
\Ee_\omega(M)$ has full topological entropy.

\subsubsection{Diffeomorphisms of tori isotopic to Anosov} \label{sec:isotopic}
The well-known DA (derived from Anosov)-example by Ma\~{n}\'{e}
(\cite{bonatti2006dynamics}) of a partially hyperbolic
diffeomorphism $f$ on the torus $\T^3$ has a unique measure of
maximal entropy. This measure is in fact a hyperbolic measure.
This example was extended by Ures~\cite{U12} to all absolutely
partially hyperbolic diffeomorphisms of $\T^3$ which are isotopic
to a linear Anosov diffeomorphism.
 In higher dimensions, Fisher, Potrie and Sambarino~\cite{FPS14} proved that every partially hyperbolic diffeomorphism of $\T^d$ which is isotopic to a linear Anosov diffeomorphism along a path of partially hyperbolic diffeomorphisms with a one-dimensional center bundle
%and  stable and unstable bundles of the same dimension as $f$
has a unique measure of maximal entropy.
 In~\cite{R16}, the last author of this paper proved that this measure is in fact a hyperbolic measure.
 Thus, it follows  from Corollary~\ref{cor:1}  that the irregular set of such systems has full topological entropy.
 Furthermore, for this type of diffeomorphisms of $\T^d$, but now assuming a compact two-dimensional center direction and $C^2$-regularity of the diffeomorphism, in \cite{Alv20} the author proved the hyperbolicity of any maximal entropy measure.
 Hence, again by Corollary~\ref{cor:1},
%these types of diffeomorphisms also have irregular sets with full topological entropy.
we may conclude full entropy of the irregular set.
% It is also proved in \cite{R16} that there exists an open set $\U$ in $\Diff^1(\T^4)$ consisting of partially hyperbolic diffeomorphisms with two-dimensional center bundle  such that  for each $f\in \U$, one can find a hyperbolic measure whose metric entropy is as close as desired  to the topological entropy of $f$.
%Thus, by Corollary~\ref{cor:1} again, the irregular set  has full topological entropy for any $f\in \U$.

\subsubsection{Trivial factors over Anosov: Diffeomorphisms of 3-dimensional nilmanifolds} \label{sec:nilmanifolds}
Let us say that a diffeomorphism $f$ of a compact manifold $M$
\emph{trivially factors over Anosov} if there exists an Anosov
map $A:\mathbb{T}^d\to \mathbb{T}^d$ on the torus $\mathbb{T}^d$, and a
continuous surjective map $\pi: M\to \mathbb{T}^d$ such that $\pi
\circ f = A\circ \pi$ and $h_{\mathrm{top}}(\pi^{-1}(y),f)=0$ for
all $y\in \mathbb{T}^d$.

\begin{prop} \label{pro-new} If $f$ is a $C^1$-diffeomorphism that trivially factors
over an Anosov $A$ then $h_{\mathrm{top}}(
I(f),f)=h_{\mathrm{top}}(f)=h_{top}(A)$.
\end{prop}
\begin{proof}
According to Bowen~\cite[Thm.~17]{bowen1971entropy} and since the
topological entropy of $f$ restricted to the fibers $\pi^{-1}(y)$ is zero it
holds that
\begin{align*}
h_{\mathrm{top}}(f) \leq h_{top}(A) + \sup_{y\in \mathbb{T}^d}
h_{\mathrm{top}}(\pi^{-1}(A),f)=h_{\mathrm{top}}(A).
\end{align*}
Moreover, since the irregular set $I(A)$ of the Anosov map $A$ has
full topological entropy (actually this is true for
any Axiom A  from Proposition~\ref{prop:1}), we obtain that
$$
    h_{\mathrm{top}}(A) = h_{\mathrm{top}}(I(A),A)
   \leq h_{\mathrm{top}}(\pi^{-1}(I(A)),f) \leq
   h_{\mathrm{top}}(I(f),f).
$$
The last inequality above follows from the inclusion
$\pi^{-1}(I(A)) \subset I(f)$, which is a straightforward
consequence of Lemma~\ref{lem:0210b}. The first inequality follows
from~\cite[Prop.~2]{PP84}. Putting together the above inequalities we
conclude the proof.
\end{proof}

%{As a first consequence of the above result we recover some of the
%previously mentioned results on the full topological entropy of
%the irregular set of partially hyperbolic DA diffeomorphisms.
%Namely, according to~\cite[proof of Thm~1.2]{U12} any absolutely
%partially hyperbolic diffeomorphism on $\mathbb{T}^d$ with $d\geq
%3$ and a one-dimensional center having quasi-isometric strong
%foliations homotopic to a hyperbolic linear automorphism trivially
%factor over an Anosov. But also, we can apply this proposition to
%the recent result of partially hyperbolic DA diffeomorphisms with
%center bundle o arbitrary dimension in~\cite[Thm.~A]{CLPV21}.}

We will now apply the above proposition to partially hyperbolic
diffeomorphisms on a $3$-dimensional nilmanifold $M\neq
\mathbb{T}^3$. Recall that $3$-dimensional nilmanifolds
%manifolds that
are circle bundles over the torus $\mathbb{T}^2$. The results
of~\cite{Ham13,HP14} imply that in this case there exists a center
foliation $\mathcal{F}^c$ for $f$ (dynamical coherence) such that
each leaf is a circle tangent to $E^c$. The foliation
$\mathcal{F}^c$ forms a circle bundle and let us consider the
quotient space $M/\mathcal{F}^c$, which is actually the $2$-torus
$\mathbb{T}^2$. Then, the map induced on the quotient space will
be a transitive Anosov homeomorphism which is conjugate to a
linear Anosov diffeomorphism which we denote by $A$. In
particular, we get that $f$ is semi-conjugate to $A$ by means of
the canonical quotient map $\pi$. Since the center leaves are
$1$-dimensional circles restricted to which $f$ is a
homeomorphism, the fibers $\pi^{-1}(y)$
 have zero topological entropy. Thus, $f$ trivially
 factors over Anosov. Consequently, from
 Proposition~\ref{pro-new} we have that any partially hyperbolic
diffeomorphism on a 3-dimensional nilmanifold (other than the
torus $\mathbb{T}^3$) has full topological entropy of the
irregular set.

In dimensions larger than three, the same argument
can be applied to partially hyperbolic $C^1$-diffeomorphisms which are dynamically coherent
and have a 1-dimensional center direction so that the center
foliation $\mathcal{F}^c$ forms a circle bundle. Such diffeomorphisms trivially factor
over Anosov (see for example~\cite[Prop.~7.1]{UVYY20}). Consequently, the
irregular set has full topological entropy.

Finally, we can
%also
apply
%immediately
the above proposition to
the recent results on partially hyperbolic DA diffeomorphisms with
center bundle of arbitrary dimension in~\cite[Thm.~A. See also
Sec.~2.3 \& Sec.~4]{CLPV21} and obtain full topological entropy of the
irregular set
%also
for this class of systems.

%See ~\cite{UVY19} where this setting was used to obtain results on
%the dichotomy of measures of maximal entropy.

\subsubsection{Perturbations of time-one map of Anosov flows}
A recent theorem of~\cite[Theorem 1.1]{BFT19} describes the following dichotomy for measures of maximal
entropy with respect to perturbations of time-one map of Anosov flows. Let $\phi^t$ be a transitive Anosov flow on a compact manifold $M$. Then, there is an open set $U$ in $\mathrm{Diff}^1(M)$ containing the time-one map $\phi^1$ in its closure, such that for any $f\in U\cap \mathrm{Diff}^2(M)$ it holds that: either (i) there are exactly two hyperbolic measures of maximal entropy or (ii) all measures of maximal entropy have zero central Lyapunov exponents. It is believed that the first case is dense and $C^2$-open in $U$ and there are preliminary announcements of these results by Crovisier and Poletti~\cite[Remark 1.2]{BFT19}. In particular, if $f$ satisfies case (i) then, as before, by using Corollary~\ref{cor:1} we conclude that the irregular set $I(f)$ has full topological entropy.

\section{Systems with  the irregular set of full Hausdorff dimension}
\label{sec3}

 We denote by $\dim_{\mathrm H} A$ the Hausdorff dimension of a set $A$ (see e.g.~\cite{Pesin98} for its definition).
 In what follows, we will study the Hausdorff dimension of $I(f)$.

\subsection{Approximation of Hausdorff dimension by $u$-conformal horseshoes}

Let $\Lambda$ be a locally maximal hyperbolic set of a $C^1$-diffeomorphism $f$ on a compact manifold. Recall that $\Lambda$ is
said to be \emph{basic set} if additionally it is a transitive set
(i.e., if it has a dense orbit). The set $\Lambda$ is said to be
\emph{$u$-conformal} if there exists a continuous function
$a^u:\Lambda \to \mathbb{R}$ such that
$Df(x)|_{E^u(x)}=a^u(x)\cdot \mathrm{Isom}_x$ for every $x\in
\Lambda$ where $\mathrm{Isom}_x$ denotes an isometry. Similarly,
$s$-conformal hyperbolic sets are defined and we say that
$\Lambda$ is \emph{conformal} if it is both $s$-conformal and
$u$-conformal. For instance, on surfaces any hyperbolic set is
conformal. According to~\cite[Theorem 22.1]{Pesin98}
and~\cite[Theorem 6.2.8]{Barreira2011}, if $\Lambda$ is a
$u$-conformal (resp.~$s$-conformal) basic set
\footnote{{Observe that due to the spectral decomposition theorem, it is enough to assume that the basic set is only transitive and not necessarily mixing, see~\cite[p. 228]{Pesin98} and  \cite[pg.86 and 123]{Barreira2011}.
}}
$$
 \mathrm{d}^u(\Lambda) \eqdef \dim_{\mathrm H} W^u_{loc}(x)\cap \Lambda
 \qquad \text{(resp.~$\mathrm{d}^s(\Lambda)\eqdef\dim_{\mathrm H} W^s_{loc}(x)\cap
 \Lambda$)}
$$
does not depend on $x\in\Lambda$.
Moreover, if $\Lambda$ is conformal then it holds that
$$
 \dim_{\mathrm H}
\Lambda = \mathrm{d}^u(\Lambda)+\mathrm{d}^s(\Lambda).
$$
This result was obtained first, independently, by McCluskey and Manning in~\cite{MM83} and Palis and Viana in~\cite{PV88} for surface diffeomorphisms.
Barreira and Schmeling in~\cite[Theorem 4.2]{BS00} use this result to prove that if $\Lambda$ is a conformal elementary horseshoe of a $C^{1+\alpha}$ diffeomorphism $f$ then,
$$
 \dim_{\mathrm H} I(f) \geq \dim_{\mathrm H} I(f)\cap \Lambda=\dim_{\mathrm H} \Lambda.
$$
Following essentially similar arguments as in~\cite{BS00} to get
the above inequality, we can get the following improvement. First,
recall that by {an \emph{elementary (hyperbolic) set} of $f$ is
understood as a topologically mixing locally maximal hyperbolic
set of some iterate $f^k$ of $f$}.

\begin{prop} \label{Dim}
Let $f$ be a $C^{1+\alpha}$-diffeomorphism of a compact manifold.
Assume that there exists a $u$-conformal elementary hyperbolic set
$\Lambda$ of $s$-index $\mathrm{d}_s(\Lambda)$
(i.e.~$\mathrm{d}_s(\Lambda)=\dim E^s_\Lambda$).
 Then
$$
\dim_{\mathrm H}  I(f)  \geq  \mathrm{d}_s(\Lambda) + \mathrm{d}^u(\Lambda).
$$
\end{prop}

\begin{proof}
{
Let $k$ be the period of the elementary set
$\Lambda$ (i.e, the smallest positive integer such that
$f^k(\Lambda)=\Lambda$).
%Let $x$ be a point in $\Lambda$. We will
%study the Hausdorff dimension of $I(f^k)$ in a small neighborhood $V$ of $x$.
%Actually,
Since $I(f^k) \subset I(f)$, we may assume for
simplicity
%in the notation
that $k=1$. Let $p$ be a fixed point in $\Lambda$.
 We will need the following fact, which we explain below: % It holds that
\begin{equation} \label{BSestimate}
\dim_{\mathrm H} W^u(p) \cap I(f|_\Lambda) = \dim_{\mathrm H} W^u(p) \cap \Lambda .
\end{equation}
Indeed, the assumption of $u$-conformality implies that the dynamics of $f$ restricted to $W^u(p)\cap \Lambda$ can be seen as a conformal expanding map.}
{
Applying~\cite[Thm.~7.5]{BS00},
(to $X=W^u(p)\cap \Lambda$, $m=1$, $g=1_X$  indicator of $X$ and $\phi_1$ to be any  H\"older continuous function on $W^u(p)\cap \Lambda$ which is not cohomologous to a constant function), we get
\[
 \dim_{\mathrm H} W^u(p)\cap I(f|_\Lambda)=\dim_{\mathrm H} I(f\vert_{W^u(p)\cap \Lambda})
 \geq  \dim I(\phi_1,f\vert _{W^u(p)\cap \Lambda})=\dim_{\mathrm H} W^u(p) \cap \Lambda
\]
where $I(\phi_1 , f\vert _{W^u(p)\cap \Lambda})$ is $\phi_1$-irregular set (see definition in Remark~\ref{Dim_cont}).
%where $I(\phi , f\vert _{W^u(p)\cap \Lambda})$ is the set of points  $z\in W^u(p)\cap \Lambda$ such that $\lim _{n\to \infty}1/n \sum _{j=0}^{n-1}\phi \circ f^j(z)$  does not exist. Fix such a $\phi$.
%Then, since $W^u(p)\cap I(f|_\Lambda) \supset I(\phi , f\vert _{W^u(p)\cap \Lambda})$ by construction, it follows from the previous equality that
%Then, since $W^u(p)\cap I(f|_\Lambda) =  I(f\vert _{W^u(p)\cap \Lambda})$, it follows from the previous equality that
%\[
% \dim_{\mathrm H} I(\phi , f\vert _{W^u(x)\cap \Lambda})
% %=
%  \dim_{\mathrm H} (W^u(p)\cap I(f|_\Lambda))
%  \geq  \dim_{\mathrm H} (W^u(p) \cap \Lambda ),
%\]
The converse inequality is obviously true and therefore we can conclude~\eqref{BSestimate}.}

{
Now there exists a point $x\in W^u(p)$ such that
\begin{equation} \label{xpeq}
\dim_{\mathrm H} W^u_{loc}(x)\cap I(f|_\Lambda)=\dim_{\mathrm H} W^u(p) \cap I(f|_\Lambda).
\end{equation}
 Fix such $x$ and let $V$ be a  small neighborhood of $x$.
%and then using (\ref{BSestimate}),
%$$\dim_{\mathrm H} I(f)\geq \dim_{\mathrm H} \left(W^u(p) \cap \Lambda \right)=.
%\dim_{\mathrm H} (W^u_{loc}(x)\cap\Lambda).$$
Let us observe
that if $y\in I(f)$, then $W^s(y)\subset I(f)$. Define the set $F$ given by the set $I(f|_\Lambda)\cap V$
saturated by the stable leaves
%$F\subset I(f)$ where
$$
F=\{z \in V: z\in W^s(y) \  \ \text{and} \  \  y \in I(f)\cap
\Lambda =I(f|_\Lambda) \}.
$$
Then $F\subset I(f)$ and as a consequence $\dim_{\mathrm H} I(f) \geq \dim_{\mathrm H} F$.
Hence, it suffices to calculate the Hausdorff dimension of $F$.}
%To do this, using the $s$-holonomy map and the orthogonal projection on $W^s_{loc}(x)$

{Consider the $s$-holonomy map $h^s:V\to W^u_{loc}(x)$  defined using the stable foliation and the transversal section $W^u_{loc}(x)$ (here ``loc'' means the local intersection with $V$ ). Let $\phi(z)=(\pi(z),h^s(z))$ for $z\in V$, where
$\pi:V\to W^s_{loc}(x)$ is the orthogonal projection on $W^s_{loc}(x)$.  Observe that $\phi(F)= W^s_{loc}(x) \times (W^u_{loc}(x) \cap I(f|_\Lambda))$,
and so $F$ can be identified with a product given
by $W^s_{loc}(x)\times (W^u_{loc}(x)\cap I(f|_\Lambda))$.
%where ``loc'' means the local intersection with $V$.
In general the $s$-holonomy is only H\"older continuous with exponent $0<\alpha\leq 1$. Then, a priori $\phi$ is just an $\alpha$-H\"older continuous map and thus $$\dim_H \phi(F) \leq \frac{1}{\alpha} \cdot \dim_H F.$$ However, since $\Lambda$ is $u$-conformal then the H\"older exponent $\alpha$ can be taken
arbitrarily close to 1 (cf.~\cite{PV88},~\cite{SS92}
and~\cite[proof of Proposition 4.6]{BCF18}). This implies that $\dim_H \phi(F) \leq   \dim_H F$.
%In general, the
%$s$-holonomy is only H\"older continuous but since we are assuming
%that $\Lambda$ is $u$-conformal the H\"older exponent can be taken
%arbitrarily close to 1 (cf.~\cite{PV88},~\cite{SS92}
%and~\cite[proof of Proposition 4.6]{BCF18}).
Thus, one gets that
$$\dim_{\mathrm H} I(f) \geq \dim_{\mathrm H} F \geq \dim_{\mathrm
H} W^s_{loc}(x) + \dim_{\mathrm H}  W^u_{loc}(x)\cap
I(f|_\Lambda) .$$
%From the proof of \cite[Theorem 4.2]{BS00}
%we have that
Using~\eqref{BSestimate} and~\eqref{xpeq},
$$\dim_{\mathrm H} I(f) \geq \dim_{\mathrm
H} W^s_{loc}(x)+\dim_{\mathrm H} W^u(p) \cap \Lambda.$$
On the other hand,
$\dim_{\mathrm H} W^u(p) \cap \Lambda =\dim_{\mathrm H} W^u_{loc}(x)\cap \Lambda=\mathrm{d}^u(\Lambda)$, see~\cite[Thm.22.1]{Pesin98}.
%$$\dim_{\mathrm H}(W^u_{loc}(x)\cap I(f|_\Lambda))=\dim_{\mathrmH}(W^u_{loc}(x)\cap \Lambda).$$
Thus, taking into account that
$\dim_{\mathrm H} W^s_{loc}(x)= \mathrm{d}_s(\Lambda)$,  we obtain
that $\dim_{\mathrm H} I(f)\geq \mathrm{d}_s(\Lambda)
+\mathrm{d}^u(\Lambda)$ concluding the proof of the proposition.
}
\end{proof}

\begin{rem}\label{rem:0410}
Let $\Lambda$ be a $u$-conformal basic set of $f$ {a
$C^{1+\alpha}$-diffeomorphism}.
 It follows from the Spectral Decomposition Theorem~\cite{KH95}
 that $\Lambda$ is a finite union of elementary sets  $\Lambda _1, \ldots , \Lambda _{k}$
 which are cyclically permuted by $f$. Thus, $\mathrm{d}_s(\Lambda)=\mathrm{d}_s(\Lambda_i)$
 and $\mathrm{d}^u(\Lambda)=\mathrm{d}^u(\Lambda_i)$ for all $i=1,\dots,  k$.
Hence, applying Proposition~\ref{Dim}, we get that
\[
\mathrm{dim} _{\mathrm{H}} I(f) \geq \max _{1\leq i\leq k} \left(
\mathrm{d}_s(\Lambda _i) + \mathrm{d}^u(\Lambda _i)\right) \geq
\mathrm{d}_s(\Lambda)+\mathrm{d}^u(\Lambda).
\]
%Note that $\dim_{\mathrm H}(W^u_{loc}(x)\cap \Lambda ) =\dim_{\mathrm H}(W^u_{loc}(x;f^k)\cap \Lambda )$, so
\end{rem}
{Proposition~\ref{Dim} with Remark~\ref{rem:0410} can be applied
to diffeomorphisms having a $u$-conformal uniformly hyperbolic
attractor $\Lambda$, which is a basic set. Since for every $y\in
\Lambda$  the unstable manifold $W^u(y)$ is contained in
$\Lambda$, then $\mathrm{d}^u(\Lambda )$ is the $u$-index $\mathrm{d}_u(\Lambda)$ of $\Lambda$ (i.e., $\dim E^u_\Lambda$). Thus, we
obtain the following corollary.
\begin{cor}\label{cor:0219b}
Let $f$ be a $C^{1+\alpha}$-diffeomorphism of a manifold $M$, having a non-trivial $u$-conformal uniformly hyperbolic attractor $\Lambda$.
Then,
%for any continuous function $\phi$,
$\dim_{\mathrm H} I(f)=\dim M $ (full Hausdorff dimension of irregular set).
\end{cor}
}
Hence, one approach to obtain full Hausdorff dimension of the irregular set is to find  a sequence of $u$-conformal horseshoes $(\Lambda_n)_{n\geq 1}$ of $f$ such that the $s$-index of $\Lambda _n$ together with the unstable Hausdorff dimension of $\Lambda _n$ approach the dimension of the whole space.

\begin{rem} \label{Dim_cont}
Given a continuous map $\phi:X \to \mathbb{R}$, we denote by
$I(\phi,f)$ the subset of  irregular points of $f$ for which  the
time average with respect to the potential $\phi$ does not
converge. This set is usually called the \emph{$\phi$-irregular set}.
Feng, Lau and Wu~\cite{FLW02} showed that if $f$ is of class
$C^{1+\alpha }$ and $  \Lambda$ is a  topologically mixing
repeller of $f$,  then for each continuous function $\phi$ such
that $I(\phi , f\vert _\Lambda ) \neq \emptyset$,
\[
\dim_{\mathrm H} I (\phi , f\vert _\Lambda ) = \dim_{\mathrm H}   \Lambda 
\]
see \cite[Theorem~1.2]{FLW02}. That is, the
{$\phi$-irregular set} has full Hausdorff dimension. The
statement also holds if $f$ is a topologically
mixing subshift of finite type~\cite[Sec.~4]{FFW01}. This can be
compared with \cite[Theorem~4.2 (1) and 7.1]{BS00} where it was
shown only for H\"older potentials $\phi$. Therefore, the proof of
Proposition \ref{Dim}
%obviously
 works to prove that %\vspace{-0.20cm}
\begin{equation*}%\label{eq:0219e}
I(\phi , f\vert _\Lambda ) \neq \emptyset \quad \Rightarrow \quad
\dim_{\mathrm H} I(\phi , f\vert _\Lambda ) \geq
\mathrm{d}_s(\Lambda) + \mathrm{d}^u(\Lambda)
\end{equation*}
for each  $C^{1+\alpha}$-diffeomorphism $f$ of a compact manifold
with a $u$-conformal elementary hyperbolic set $\Lambda$ of
$s$-index $\mathrm{d}_s(\Lambda)$. Consequently, as in Corollary~\ref{cor:0219b}, we get full Hausdorff dimension of any non-empty $\phi$-irregular sets of a non-trivial $u$-conformal uniformly hyperbolic attractor of a $C^{1+\alpha}$-diffeomorphism. \vspace{-0.20cm}
\end{rem}

\subsection{One dimensional examples}
For one-dimensional uniform expanding dynamics and, in general, for conformal repellers in higher dimension,
it  follows from~\cite{BS00,FFW01,FLW02} that the irregular set  has full Hausdorff dimension (on the repeller).
This result was extended in~\cite{Chung2010} for topologically exact $C^2$-interval maps admitting an absolute continuous invariant probability measure (see Proposition 8  and the remark before section 4 of \cite{Chung2010}).
Another extension was made in \cite{MY17} for a class of  non-uniformly expanding one-dimensional dynamics which includes  Manneville-Pomeau maps, that is, the map $T : [0,1] \to [0,1]$ defined by $T(x) = x + x^{1+\alpha} \mod 1$, where $0 < \alpha < 1$. \vspace{-0.20cm}
%\marginpar{\tiny To-do: Include Chung's paper as a reference}

\subsection{Surface diffeomorphisms}
{For $C^{1+\alpha}$ surface diffeomorphisms any hyperbolic set is
$u$-conformal and in particular Proposition~\ref{Dim} and
Corollary~\ref{cor:0219b} can be applied.}
%Proposition~\ref{Dim} with Remark~\ref{rem:0410} can be applied to
%surface diffeomorphisms having a uniformely hyperbolic attractor $\Lambda$, which is a basic set. Since for every $y\in \Lambda$  the unstable manifold $W^u(y)$ is contained in $\Lambda$, then $\mathrm{d}^u(\Lambda )=1$.
%Thus, we obtain the following corollary.
%\begin{cor}\label{cor:0219b}
%Let $f$ be a $C^{1+\alpha}$-diffeomorphism of a compact surface $M$, having a non-trivial uniformely hyperbolic attractor $\Lambda$.
%Then,
%for any continuous function $\phi$,
%$\dim_{\mathrm H} I(f)=2$ (full Hausdorff dimension of irregular set).
%\end{cor}
Homoclinic tangencies are the obstruction to hyperbolicity for surface diffeomorphisms.
In \cite[Theorem 1.6]{DN05}, Downarowicz and Newhouse  constructed a residual subset $\mathcal R$ of $\mathrm{Diff}^r(M)$ ($r\geq 2$) for a compact surface $M$,  such that if $f\in \mathcal R$ and $f$ has a homoclinic tangency, then $f$ has a sequence of elementary horseshoes $(\Lambda_n)_{n\geq 1}$ with
$\dim_{\mathrm H} \Lambda_n \to 2$ as $n\to \infty$.
%Indeed, this is a slight modification of Theorem 1.6 in \cite{DN05}, using the claim in page 56 of the article.
Therefore, combining this with Corollary~\ref{thm-entropy-two-dimensional} and Proposition \ref{Dim}, we obtain full entropy and Hausdorff dimension of irregular set in this setting.
%Since our dynamics are in dimension two, by Corollary  \ref{thm-entropy-two-dimensional} we conclude the following.
\begin{cor}\label{cor:0219a}
There exists a residual subset $\mathcal R$ of
$\mathrm{Diff}^r(M)$  for a compact surface $M$  with $r\geq 2$
such that if $f\in \mathcal R$ and $f$ has a homoclinic tangency,
then $\dim_{\mathrm H} I(f) =2$. In particular, the irregular set
of $f$ has full topological entropy and full Hausdorff dimension.
\end{cor}

\subsection{Parametric families in low dimensions}\label{ss:0219a}
\subsubsection{Quadratic family}
As previously mentioned, it follows from \cite{Chung2010} that if $f$ is a topologically exact $C^2$-interval map admitting an absolute continuous invariant probability measure, then $\dim_{\mathrm H} I(f)=1$.
On the other hand, it is well known that these conditions are satisfied for the quadratic map $f_a(x)=x^2+a$ for parameters $a$ in a Lebesgue positive measure set $E$ close to $a=2$ (cf.~\cite{Jak81}).
Thus, the irregular set of the quadratic map $f_a$ has full Hausdorff dimension for all $a\in E$.
% (full Hausdorff dimension).

\subsubsection{The Standard map}
The standard family  is defined on the torus by
$$
f_k(x,y)=(-y+2x+k\sin(2\pi x), x) \quad \mod \mathbb Z^2.
$$
Gorodetski \cite{Gor12} proved that for each $k$   in  a residual
set of $[k_0, \infty )$ with a  large $k_0$, there exists a
sequence of elementary horseshoes $(\Lambda_n^k)_{n\geq 1}$ of
$f_k$ satisfying $\dim_{\mathrm H} \Lambda_n^k \rightarrow 2$ as
$n\to \infty$. Thus, again  by Proposition~\ref{Dim} we conclude
that  $\dim_{\mathrm H} I(f_k)=2$ for all $k$   in  the residual
set (full Hausdorff dimension). We note that the Gorodestki's
approximation theorem by horseshoes holds for every
area-preserving one-parameter  family $(g_t) $ generically
unfolding a quadratic homoclinic tangency at $g_0$.
%With respect entropy of the irregular set, recently Obata \cite{Ob20} proved that for large enough parameter $k$ the standard map has a unique measure of maximal entropy.
%Again since this is in dimension $2$, applying
%Theorem~\ref{KatokType}, we obtain that $h_{\mathrm{top}}(f_k)=
%h_{\mathrm{top}}(I(f_k),f_k)$, or in other words the irregular set has full
%entropy.

%Thus, we may conclude the following:
%
%\begin{cor}
%There exist $k_0>0$ (large enough) and a residual set
%$\mathcal{R}$ of parameters in $[k_0,\infty)$  such that for every
%$k\in \mathcal{R}$, the irregular set of the standard map $f_k$,
%has full entropy and full Hausdorff dimension.
%\end{cor}

%\section{Invariant graphs}
\subsection{Examples in skew-products over a symbolic shift}\label{ss:0219b}
\subsubsection{Hyperbolic graphs}
Consider a skew product $F$ of the form $F(x, y)=(f(x), g_x(y))$ defined in $\Lambda\times
\mathbb{R}^n$, where $\Lambda$ is an elementary  horseshoe of a two-dimensional $C^2$-diffeomorphism $f$.
As mentioned before, from \cite{BS00} we have that~$\dim_{\mathrm H} (I(f)\cap \Lambda )=\dim_{\mathrm H} \Lambda $.
Let us further observe the following.
Given $x_0\in I(f)$,  there exists a continuous function $\phi_0: \Lambda\rightarrow \mathbb{R}$ for
which the Birkhoff time average
$$\lim _{n\to\infty} \frac{1}{n}\sum_{j=0}^{n-1} \phi_0(f^j(x_0))$$
does not exist.
%with respect to $x_0$.
Considering $\tilde{\phi}_0(x,y)=\phi_0(x)$ and because of the skew-product structure of $F$, the above implies that for every $y\in \mathbb{R}^n$, the Birkhoff time average with respect to the point $(x_0, y)$ also does not converge.
%is  in $I(\tilde{\phi}_0, F)$, where $\tilde{\phi}_0(x,y)=\phi_0(x)$.
Therefore, we obtain in this case full Hausdorff dimension of the irregular set of $F$:
\begin{lem}\label{lem:0219d}
$I(f)\times \mathbb{R}^n\subset I(F)$ and $\dim_{\mathrm H} I(F) =\dim_{\mathrm H} \Lambda+ n$ (full Hausdorff dimension).
\end{lem}

Thus, in this setting, it makes more sense to ask the following.
Is it true that if $I(\phi, F)$ is non-empty, then it has full Hausdorff dimension?
Recall that given a continuous function $\phi: \Lambda\times
\mathbb{R}^n \to \mathbb{R}$, the $\phi$-irregular set $I(\phi, F)$ is defined as
%For a continuous map $T: X\to X$ on a metric space and a continuous function $\phi : X\to \mathbb R$, let $I(\phi ,T)$ be
the set of points $(x,y)\in \Lambda\times \mathbb{R}^n$ for which   $\lim _{n\to\infty} n^{-1}\sum_{j=0}^{n-1} \phi (F^j(x,y))$ does not exist.

Assume from now that the fiber map $g_x$ is a contraction for every $x\in \Lambda$.
Then, the maximal invariant set of $F$ is unique and is the graph $\{ (x, \Phi (x)) \mid x\in \Lambda\}$ of an invariant map $\Phi:\Lambda\rightarrow \mathbb{R}^n$, that is $g_x(\Phi(x))=\Phi(f(x))$ for every $x\in \Lambda$.
This class includes an important example of horseshoes called \emph{blenders} (cf.~\cite{bonatti2006dynamics}).
We emphasize that compared with the one-dimensional or two-dimensional cases, in higher dimensions, there is no explicit way to calculate the Hausdorff dimension of the maximal invariant set.
Let us mention the work \cite{DGGJ19}, where D\'iaz et al. calculated  the box dimension of the maximal invariant set in dimension $3$ of a variety of examples of skew-products with contracting/expanding fibers, including blenders.
Nevertheless,
%we have seen   in Lemma \ref{lem:0219d} that the irregular set of skew-product maps has  full Hausdorff dimension, and
in the special case when the fiber maps are contractions, we can give a complete answer to the above question about
the Hausdorff dimension of
 $\phi$-irregular sets:
% in the particular case of contracting fiber maps.

\begin{prop}
Assume that $g_x$ is a contraction for all $x\in \Lambda$.
Then, for any continuous function $\phi :\Lambda \times \mathbb R^n \to \mathbb R$, either $I(\phi, F)=\emptyset$ or $\dim_{\mathrm H} I(\phi ,F)=\dim_{\mathrm H} \Lambda + n$.
% (full Hausdorff dimension).
\end{prop}

\begin{proof}
Denote by $\Phi : \Lambda \to \mathbb R^n$
%: \Lambda \to \mathbb R^n$
 the invariant map of $F$ and by $\Gamma$  the graph of $\Phi$.
Let $\phi $ be a continuous function on $\Lambda \times \mathbb R^n$.
Then it induces a continuous function $\tilde\phi$ in $\Lambda$, defined by $\tilde\phi(x)=\phi(x, \Phi(x))$.
We claim that
\begin{equation} \label{eq:pp}
I(\tilde{\phi},f|_{\Lambda}) \times
\mathbb{R}^n \subset I(\phi,F).
\end{equation}
Indeed, if $(x,y)\in I(\tilde{\phi},f|_{\Lambda}) \times \mathbb{R}^n$ then $(x,\Phi(x)) \in I(\phi, F)$.
Hence, since $\{x\}\times \mathbb{R}^n =W^s_{loc}(x,\Phi(x)) \subset I(\phi,F)$ we get that $(x,y)\in I(\phi,F)$.
That is, \eqref{eq:pp} holds.

Suppose that $I(\phi, F)$ is non-empty and take $(x, y)\in I(\phi,F)$.
 Observe that $(x,y)\in W^s_{loc}(x, \Phi(x))$.
 Then $(x,\Phi(x))\in I(\phi, F|_\Gamma)$ and so $I(\phi, F|_\Gamma)$ is non-empty.
 Thus $I(\tilde\phi, f|_{\Lambda})$ is also non-empty.
Since $\Lambda$ is a two-dimensional horseshoe, it follows (see Remark~\ref{Dim_cont}) that $I(\tilde\phi, f|_{\Lambda})$ has full Hausdorff dimension.
Combining this with \eqref{eq:pp} we obtain that
$$
\dim_{\mathrm H} I(\phi,F) = \dim_{\mathrm H} \Lambda +n
$$
completing the proof.
\end{proof}

%\subsubsection{Bony graphs}
\subsubsection{Porcupines and bony graphs}  Porcupines were defined in~\cite{GD12} and form an important class of partially hyperbolic sets, mixing contracting and expanding behaviour. Below we describe the setting borrowing the notation from \cite{DM}.
Let $\Sigma_2=\{0,1\}^\mathbb{Z}$ and consider a skew-product map $F$ of $\Sigma_2\times [0,1]$ defined by
$$
F(\xi, x)=(\sigma (\xi), f_{\xi_0}(x))
$$
where $\sigma:\Sigma_2 \to\Sigma_2$ denotes the shift map and $\xi_0$ is the zeroth coordinate of $\xi=(\xi_i)_{i\in\mathbb{Z}}$.
We actually study a one-parameter family $(F_t)$ given by $F_t(\xi, x)=(\sigma (\xi), f_{\xi_0, t}(x))$, where $f_{0,t} = f_0$ is an increasing concave $C^2$-map independent of $t$ with two fixed hyperbolic points $f_0(0) = 0$, $f_0 (1) = 1$ and $f_{1,t}$ is the explicit affine map $f_{1,t} (x) = t(1-x)$.

Denote by $\Lambda_t$ the maximal invariant set of $F_t$.
The set $\Lambda_t$ is semi-conjugate to the shift map in $\Sigma_2$.
Namely, if $\Pi$ denotes the projection of $\Sigma_2\times [0,1]$ onto the first coordinate, then $\Pi_t \circ F_t = \sigma\circ\Pi_t$.
For each $\xi\in \Sigma_2$ we consider the set
$$
\Pi_t^{-1}(\xi)\eqdef \Pi^{-1}(\xi)\cap \Lambda_t=\{\xi\}\times I_\xi
$$
 called a \emph{spine} of $\Lambda_t$.
 Here $I_\xi$ is an interval of $[0,1]$.
The spine $I_\xi$ is said to be \emph{non-trivial} if it is not a singleton and \emph{trivial} otherwise.
In this way, we split the set $\Sigma_2$  into two disjoint invariant sets $\Sigma^{\mathrm{non}}_{2,t}$ and $\Sigma_{2,t}^{\mathrm{trv}}$ consisting of sequences with non-trivial and trivial spines, respectively.
Moreover, $\Lambda_t$ is a \emph{porcupine-like horseshoe} (shortly, a \emph{porcupine}), that is, a topologically transitive set of $F_t$  such that  the sets $\Sigma_{2,t}^{\mathrm{non}}$ and $\Sigma_{2,t}^{\mathrm{trv}}$ are both dense and uncountable in $\Sigma_2$.
We observe that the topological entropy of $F_t$ is always $\log(2)$ because $F_t$ is semi-conjugate to the full shift $\Sigma_2$ and the fiber maps do not have critical points.

%The topological entropy of $F_t$ is $\log(2)$.

In the set $\Sigma_2$ consider the distance defined by
$$
d(\xi, \zeta) = 2^{1/2} 2^{-|n|} \quad \text{where $\vert n\vert$  is the smallest value
with $\xi_n\not = \zeta_n$.}
$$
With this distance the Hausdorff dimension of $\Sigma_2$  is $2$.
Denote by $\mathbf{b}_{1/2}$ the $(1/2, 1/2)$-Bernoulli measure on $\Sigma_2$, which in this case coincides with the $2$-dimensional Hausdorff outer measure $m_2$.
It is shown in \cite[Theorems 1 and 2]{DM} that, for every $t\in (0,1)$,
$$
\dim_{\mathrm H}(\Sigma^{\textrm{non}}_{2,t})<2 \qquad \text{and} \qquad \mathbf{b}_{1/2}(\Sigma^{\mathrm{trv}}_{2,t})=m_2(\Sigma^{\mathrm{trv}}_{2,t})=1.
$$

Consider the following set
$$
G_t=\{ (\xi, g(\xi)) \in \Lambda_t: \xi \in \Sigma^{\mathrm{trv}}_{2,t}\}
$$
of a $\mathbf{b}_{1/2}$-almost everywhere defined function $g$ from the base of the skew product to the fiber.
In particular, one gets that $\Lambda_t$ is a union of the almost-everywhere defined graph $G_t$ with some non-trivial spines (called bones),
%which are segments in the fibers over those points of the base where this function is not defined.
and this is termed a \emph{bony graph}.

\begin{prop}
%With respect to the above map $F_t$ and a
For every  continuous
%real-valued function
$\phi$ on $\Sigma _2 \times [0,1]$ and  $t\in (0,1)$, either $I(\phi, F_t)\cap G_t=\emptyset$ or
%$$h_{\mathrm{top}}(I(\phi, F_t))=h_{\mathrm{top}}(F_t) \quad  \text{and} \quad
$\dim_{\mathrm H}(I(\phi, F_t))\geq 2.$ In the latter case $I(\phi, F_t)$ also has full entropy.
\end{prop}
\begin{proof}
Consider the continuous function $\tilde{\phi}_t= \phi \circ \Pi^{-1}_t$, which is a priori well-defined only in $\Sigma^{\mathrm{trv}}_{2,t}$ but can be extended to a continuous function on $\Sigma_2$.
If $I(\phi,F_t)\cap G_t\neq\emptyset$, then
%there exists a (Holder)-continuous function $\tilde{\phi}$ defined in $\Sigma_2$ so that
%$\phi(\xi, x)=\tilde\phi(\xi)$ for $(\xi, x)\in G_t$ and
$I(\tilde\phi_t, \sigma)\neq\emptyset$ and so it follows (Remark~\ref{Dim_cont}) that $I(\tilde\phi_t, \sigma)$ has full Hausdorff dimension.
From \cite[Theorem 1]{DM}, we have that $\dim_{\mathrm H}(\Sigma^{\mathrm{non}}_{2,t})<2$.
%and $m_2(\Sigma^{trv}_{2,t})=1$.
As $\Sigma^{\mathrm{non}}_{2,t}\cup \Sigma ^{\mathrm{trv}}_{2,t}=\Sigma_2$, therefore, $I(\tilde\phi _t, \sigma)\cap \Sigma ^{\mathrm{trv}}_{2,t} $ also has full Hausdorff dimension. In the case of the symbolic shift $\Sigma_2$, full Hausdorff dimension also implies full entropy of the set $I(\tilde\phi _t, \sigma)\cap \Sigma ^{\mathrm{trv}}_{2,t}$, which is $\log(2)$ and coincides with the entropy of $F_t$.
%with respect to the shift on $\Sigma_2$.

Consider the graph over $I(\tilde\phi _t, \sigma)\cap \Sigma ^{\mathrm{trv}}_{2,t} $,
$$
\tilde{G_t}=\{ (\xi, x) \in \Lambda_t\, : \,\xi \in  I(\tilde\phi _t, \sigma)\cap \Sigma ^{trv}_{2,t} \}.
$$
Then
%$h_{\mathrm{top}}(\tilde{G_t})=log(2)$ and
$\dim_{\mathrm H}(\tilde{G_t})\geq \dim_{\mathrm H}(I(\tilde\phi_t, \sigma)\cap \Sigma ^{\mathrm{trv}}_{2,t}) \geq 2$ and $h_{\mathrm{top}}(\tilde{G_t})=\log(2)$.
Since by construction $\tilde{G_t}\subset I(\phi, F_t)$,
we reach the conclusion required for $I(\phi, F_t)$.
\end{proof}

\begin{question}
Using the fact that there is nonuniform contraction along the trivial fibers is it possible to show that actually $\dim_{\mathrm H}(I(\phi, F_t))=3$ (full Hausdorff dimension) for $F_t$ in the above proposition? Do similar statements hold for the skew-products over the shift with circle fibers studied in \cite{DGM17}, the bony attractors in the results of Kleptsyn and Volk \cite{KV14}, and the higher dimensional bony-attractors of Kudryashov \cite{Kudryashov2010}?
\end{question}

\subsection{mostly contracting systems}\label{mostlycontracting}
Consider a $C^{1+\alpha}$ partially hyperbolic diffeomorphism $f$
on a compact manifold $M$. Recall this means that the tangent
bundle $TM$ splits in three $Df$-invariant subbundles $E^s\oplus
E^c\oplus E^u$, where $E^s$ is uniformly contracting, $E^u$ is
uniformly expanding and $E^c$ called the central bundle is
in-between. The extremal subbundles, $E^s$ and $E^u$, can be
uniquely integrated obtaining the so-called strong stable and
unstable foliations of $M$. A \textit{Gibbs $u$-state} is an
invariant probability measure whose conditional probabilities
along strong unstable leaves are absolutely continuous with
respect to the Lebesgue measure on the leaves
(see~\cite{bonatti2006dynamics} for an introduction). Gibbs
u-states are important in the study of physical measures and in
particular, every ergodic Gibbs $u$-state with negative center
Lyapunov exponents is a physical measure. It said that $f$  has
\textit{mostly contracting center} if all of its Gibbs u-states
have only negative center Lyapunov exponents. These systems were
first studied in~\cite{BV00} where the existence and finiteness of
physical measures were shown. The key property for us is that the
Pesin formula is satisfied with respect to an ergodic Gibbs
u-state $\mu$: the metric entropy $h_\mu(f)$ is equal to the sum
of positive Lyapunov exponents (counting with multiplicity).

Let $\mathrm{Diff}^{1+}(M)$ denote the set of all diffeomorphisms which are $C^{1+\alpha}$ for some $\alpha$. This set can be endowed with the $C^1$-topology, and systems with mostly contracting center form an open set in $\mathrm{Diff}^{1+}(M)$ with this topology~\cite{Yang18}.
One can similarly define systems having a \textit{mostly expanding center} and analogous properties hold for them, including the Pesin formula (\cite{ABV00, Yang18, HUY19}).

\begin{thm}\label{Contract} Let $f$ be $C^{1+\alpha}$ partially-hyperbolic diffeomorphism with mostly contracting center and assume $\dim(E^u)=1$. Then
there exists
an open neighborhood $\mathcal{U}$ of $f$ in $\mathrm{Diff}^{1}(M)$ and
a residual set $\mathcal{R}\subset \mathcal{U}$ such that for $g\in\mathcal{R}$,
the irregular set has full Hausdorff dimension, that is $\dim_H I(g)=\dim M$.
\end{thm}

%\subsubsection{Remarks on possible extensions of Theorem~\ref{Contract}}
The proof of the above theorem is given in Section~\ref{Contractproof}. Here we will describe the key steps, which are as follows.
\begin{enumerate}
\item Existence of a dense set of maps in $\mathcal{U}$ having a hyperbolic measure satisfying the Pesin Formula.
\item Approximating the measure in entropy by a horseshoe and afterwards perturbing the horseshoe in $C^1$-topology to the so-called standard affine horseshoe.
\item Projecting the horseshoe in the unstable direction, we can obtain an attractor of an affine iterated function system on the line. This permits to calculate the Hausdorff dimension of this set, which is going to be relatively big because of the Pesin formula.
\item Calculating the dimension of the irregular set of the horseshoe.
\end{enumerate}

\subsubsection{Remarks on possible extensions of Theorem~\ref{Contract}}
We are interested in extending Theorem~\ref{Contract} to
diffeomorphisms with mostly expanding center or conservative
systems in dimension 3. The first three of the above steps
actually still hold in these cases. Except that now we would be
dealing with an attractor of an affine iterated function system
\textit{in the plane} as the exponent in the center direction will
be also positive. Under some algebraic assumptions, it is possible
to conclude that the Hausdorff dimension of the attractor is equal
to the affinity dimension for a system of affine contractions in
the plane with a diagonal linear part
(\cite[Proposition~7.2]{MS19}). Then, one can calculate the
affinity dimension using, for example, the formulas in \cite{Mor19}.
The main problem is that it is unknown how to calculate the
Hausdorff dimension of \textit{the irregular set} for these
iterated systems. Since they are non-conformal, we cannot apply
directly~\cite{BS00}, as was done in Proposition~\ref{Dim}. This
leads to the following question.

\begin{question}
Consider a planar contractive, affine, diagonal iterated function system as in \cite{MS19}(or \cite{PW94}) for which the Hausdorff dimension of the attractor is known explicitly. Show that the irregular set has full Hausdorff dimension.
\end{question}

\subsection{$C^1$-generic diffeomorphisms with no dominated splitting}\label{nodomination}
Motivated by the Newhouse's theorem on approximation by horseshoes with large Hausdorff dimension for surface diffeomorphisms with homoclinic tangencies~\cite{Newhouse78}, Buzzi, Crovisier and Fisher have recently obtained a similar result in higher dimension in~\cite[Theorem 6]{BCF18}.
Using this extension and Proposition~\ref{Dim} we can obtain the next result on the full dimension of the irregular set.
To state the result precisely, we need some notation.
Each periodic point $P$ of a diffeomorphisms $f$ of a compact manifold of dimension $d$ admits Lyapunov exponents $\lambda_1(f, P)\geq \dots \geq \lambda_d(f,P)$, listed with multiplicity.
We set $\lambda^+_i= \max\{\lambda_i, 0\}$ and $\lambda^-_i= \max\{-\lambda_i, 0\}$.
Also we denote
$$
\Delta^+(f,P)= \sum_{i=1}^d \lambda^+_i \qquad \text{and} \qquad
\Delta^-(f,P)= \sum_{i=1}^d \lambda^-_i.
$$
\begin{thm}\label{thm:0219c}
For a  $C^1$-generic diffeomorphism $f$ of a compact manifold $M$, if $f$ has a periodic point $P$ such that  the homoclinic class $H(P)$ of $P$ has no dominated splitting and $\Delta^-(f,P)\geq \Delta^+(f,P)$, then  $\dim_{\mathrm H} I(f)= \dim M$.
\end{thm}
See Section~\ref{thm:0219c_proof} for the proof.

\section{Geometric Lorenz flows}\label{s:Lorenz}
In this section, we show that every geometric Lorenz flow has the irregular set of full topological entropy and full Hausdorff dimension.
%See \cite[Section 2]{MPR20} for a standard construction of a geometric Lorenz flow.

The irregular set $I(f)$ of  a continuous flow $f=\{f^t\}_{t\in\mathbb{R}}$ on a compact metric space $X$ is defined as the set of points $x\in X$ such that there exists a continuous function $\phi : X\to \mathbb R$ for which the time average
\begin{equation}\label{eq:0106}
\lim _{T\to\infty} \frac{1}{T} \int ^T_0 \phi \circ f^t (x) \,dt
\end{equation}
does not exist. We will also use the set $I(\phi , f)$ defined as
the set of points $x\in X$ such that \eqref{eq:0106} does not
exist. Recall that the Abramov-Bowen
formula~(\cite[Proposition~21]{bowen1971entropy}; see also
\cite{A59} as its origin for measurable flows) establishes that
\[
h_{\mathrm{top}}(f^1)=\frac{1}{|t|} h_{\mathrm{top}}(f^t) \quad
\text{for all $t\not= 0$}.
\]
So, the topological entropy $h_{\mathrm{top}}(f)$ of $f$ is
defined as the topological entropy of the time one map $f^1$,
i.e.~$h_{\mathrm{top}}(f)= h_{\mathrm{top}}(f^1)$. We also refer
to~\cite[p.~185]{BR75} for an equivalent definition of the
topological entropy of $f$ given as a straightforward analogue of
the usual (i.e.~Bowen-Dinaburg's) topological entropy for a
continuous map on a compact metric space.

In order to define the topological entropy $h_{\mathrm{top}}(A, f)$ of $f$ on a (not necessarily compact) invariant set $A$ of a totally bounded space $X$, we have two alternatives in literature.
The first alternative is to extend the definition of the Pesin-Pitskel'  topological entropy for maps \cite{PP84} to flows,
%as a Carath\'eorody characteristic of dimension type (as Pesin and Pitskel'~\cite{PP84} did for maps),
which was done by Thompson \cite[Section 4]{T10} and Shen-Zhao \cite{SZ12}.
%using the theory of Carath\'eorody structure for maps
%(easily adapted to flows, see~\cite{Pesin98} and \cite[Section 4]{T10}).
For this definition of the topological entropy for flows, the Abramov-Bowen type formula is also proven, i.e.
\[
h_{\mathrm{top}}(A,f)=\frac{1}{|t|}h_{\mathrm{top}}(A,f^t) \quad \text{for all $A\subset X$ and $t\not=0$}.
\]
The other alternative is,  as Pacifico and Sanhueza~\cite{PS19} recently did, to extend  Bowen-Hausdorff version of the topological entropy for maps \cite{Bow73} to flows. They also showed  the Abramov-Bowen type formula for this topological entropy.
When $A$ is compact, as we notified in Section \ref{sec2} all the previously-mentioned versions  of the topological entropy of a continuous map on $A$ coincide, and so is a continuous flow on $A$.

To state our main result clearly,
we employ the construction of geometric Lorenz flows of \cite{MPR20}.
Given a $C^1$ vector field $V$ of $\mathbb R^3$, we denote by $f_V=\{f^t_V\} _{t\in \mathbb R}$ the continuous flow generated by $V$.
We recall that in the setting of \cite[Section 2]{MPR20},
a geometric Lorenz flow $f_V$ generated by a   $C^1$   vector field $V$  satisfies that
%which defines a geometric Lorenz attractor $\Lambda$.
%In particular, i
\[
\text{$V(x,y,z) =(\lambda _1x,\lambda _2 y,\lambda _3z)$ on $[-1,
1]^3$} \quad \text{with $0<-\lambda _3<\lambda _1<-\lambda _2$.}
% \quad \text{and} 0<\alpha =- \frac{\lambda _3}{\lambda _1} <1 <\beta.
\]
Moreover, the Poincar\'e map $P: S_0 \to S$ of $f_V$ with $$ S=[-1,1]^2 \times \{1\} \quad \text{and} \quad S_0 =\{ (x,y,1) \in S \mid x\neq 0\}$$ is
%well-defined and
 a skew-product map of the form
 \begin{equation}\label{eq:0905}
 \text{$P(x,y,1)=(F(x),G(x,y) ,1)$ for $x\in[-1,1]\setminus \{0\}$ and $y\in [-1,1]$ }
 \end{equation}
 where $F:[-1,1]\setminus \{0\} \to [-1,1]$ is a $C^2$ piecewise expanding map and $y\mapsto G(x,y)$ is a contraction map for every $x\in [-1,1]\setminus \{0\}$.
%The maximal invariant set $\Lambda _W:= \bigcap _{t\in\mathbb R}f^t_W(D)$ is a transitive locally maximal invariant singular-hyperbolic set of $f_W$ (i.e.~$f_W$ is partially hyperbolic with volume expanding center-unstable subbundle  on $\Lambda _W$ and all singularities of $W$ in $\Lambda _W$ are hyperbolic)
Let $K$ be the closure  of $\{ f^t_V(x) \mid x\in S_0, \; t\geq 0\}$ and
%Then, the maximal invariant set
$$ \Lambda _V := \bigcap _{t\geq 0}f^t_V(K).$$
%(called the \emph{geometric Lorenz attractor} of $f_V$).
%,  is known to be  transitive and singular-hyperbolic   (i.e.~$f_V$ is partially hyperbolic with volume expanding center-unstable subbundle  on $\Lambda _V$ and all singularities of $V$ in $\Lambda _V$ are hyperbolic).
%The striking fact is that the geometric Lorenz attractor is robust:
Then, one can take an open set $D$ including $ \Lambda _V $ such that for any $C^1$-close vector field $W$ to $V$,  the set $\Lambda _W:= \bigcap _{t\in\geq 0}f^t_W(D)$ is transitive and
singular-hyperbolic.
That is, $f_W$ is partially hyperbolic with volume expanding center-unstable subbundle  on $\Lambda _W$ and all singularities of $W$ in $\Lambda _W$ are hyperbolic.
The set $\Lambda_W$ is called the \emph{geometric Lorenz attractor} of $f_W$ while  $D$ is said to be a \emph{trapping region}.
Finally,
% and $\Lambda _W$ the geometric Lorenz attractor of $f_W$.
we denote by $\mathcal X^1(D)$ the space of $C^1$ vector fields on $D$.

\begin{thm}\label{prop:0904}
Let $V$ be a $C^1$ vector field which defines a geometric Lorenz flow $f_V$ with a trapping region $D$.
%Assume that
%\begin{equation}\label{eq:0106c}
%-\lambda _3 < -\lambda _2  - 2\lambda _1.
%%-\lambda _2 > -\lambda _3 + 2\lambda _1.
%%-\frac{\lambda _3}{\lambda _1} +2 < - \frac{\lambda _2}{\lambda _1}.
%\end{equation}
Then, there is a neighborhood $\mathcal N\subset \mathcal X^1(D)$
of $V$  such that  for any $W\in \mathcal N$,
\begin{itemize}
\item[$\mathrm{(i)}$] $h_{\mathrm{top}} (I(f_W), f_W) = h_{\mathrm{top}} (f_W)$ \quad (full topological entropy).
\end{itemize}
Furthermore, if $- \frac{\lambda
_2}{\lambda _1}
> -\frac{\lambda _3}{\lambda _1} + 2$, then
%there exists a $C^1$-neighborhood $\mathcal N$ of $V$  such that
  for any $W\in \mathcal N$,
\begin{itemize}
\item[$\mathrm{(ii)}$] $\dim_{\mathrm H} I(f_W)  =3$ \quad (full Hausdorff dimension).
\end{itemize}
\end{thm}

In \cite{KirLiSo} (see
also~\cite{yang2020historical}), it was proven that the irregular
set of any geometric Lorenz flow is residual in the trapping
region. More precisely, they constructed a residual set
$\mathcal{R}$ consisting of points whose partial time averages for
some continuous function oscillate between the Dirac measures of
the singularity and a periodic orbit. However, since both of the
measures have zero topological entropy and zero Hausdorff
dimension, it seems that the topological entropy and Hausdorff
dimension of $\mathcal{R}$ is zero. (cf.~\cite{CZ13})}

\begin{rem}
The condition  $- \frac{\lambda _2}{\lambda _1} > -\frac{\lambda_3}{\lambda _1} + 2$ in Theorem \ref{prop:0904} is only used to ensure the regularity of the quotient map along the strong stable foliation of the  Poincar\'e map $P_W$ of $f_W$.
This regularity is automatically satisfied when $W=V$ by construction (we will explain it more precisely in~\S\ref{ss:LH}).
Thus,  one gets that $\dim_{\mathrm H} I(f_V) =3$ without the extra condition.
\end{rem}

As in the map case, the key ingredient is  an approximation theorem  by horseshoes, so before starting the proof of Theorem~\ref{prop:0904}, we  recall the definition of horseshoes for continuous flows.
Given a homeomorphism $T: X \to X $ on a compact metric space $X$ and a continuous function $\rho : X \to \mathbb R_+:= \{ z>0\}$, the \emph{suspension flow} $T_\rho=\{ T^t_\rho\} _{t\in \mathbb R}$ of $T$ over $\rho$ is defined as
\[
T^t_\rho : X_\rho \to X_\rho,  \qquad T^t_{\rho}( x,z)=( x, z+t)
\]
where $X _\rho$ is the $\rho$-suspension space  given by
\[
X  _\rho = \left\{ (  x, z) \mid   x \in X, \;  0\leq z \leq
\rho(  x) \right\} / \sim
\]
with $(  x,\rho(  x)) \sim (T(  x),0)$.
% ($\rho$ is called the \emph{roof function} of the suspension flow).
%The function $\rho$ is called the roof function of the suspension flow
%It is straightforward to see that
%%\begin{equation}\label{eq:0209c}
%$
%I(f|_{\Lambda})=\pi (I(T_\rho)).$
%%\end{equation}
%Furthermore,
For a continuous flow $f=\{f^t\}_{t\in \mathbb R}$,  a compact $f$-invariant set $\Gamma$ is called a \emph{horseshoe} of $f$, if there exists a suspension flow $\sigma _\rho=\{ \sigma ^t_\rho\} _{t\in \mathbb R}$ of the left-shift operation $\sigma :\Sigma \to \Sigma$ of a two-sided topologically mixing subshift of finite type with a finite alphabet over a continuous function $\rho$, and a homeomorphism $\pi :\Sigma _\rho \to \Gamma$ such that
\[
f^t \circ \pi = \pi \circ \sigma _\rho ^t \quad \text{for every $t\in \mathbb R$}.
\]
%\color{blue}(Our definition of horseshoe for flows is slightly different from \cite{LSWW20}: they assumed $\Sigma$ to be a full shift, but the difference makes no influence on this paper.) \color{black}
%On the other hand, given a continuous
% function $\phi :\Sigma_\rho \to \mathbb{R}$ note that $I
%(\phi,\sigma_\rho)\neq \emptyset $ if and only if $I (\phi \circ
%\pi^{-1} ,f )\neq \emptyset$.
We will use the following preliminary lemma, whose proof  will be given in~\S\ref{p:55}.

\begin{lem}\label{lem:0904}
Let $\Sigma \subset \{1, 2,\ldots , N\} ^{\mathbb Z}$ be  a topologically mixing subshift of finite type with $N\geq 2$ and $\sigma : \Sigma \to \Sigma$  the left-shift operation.
Let $\rho : \Sigma \to \mathbb R_+$ be a continuous function.
Then, there exists a continuous function $\phi :\Sigma_\rho\to \mathbb R$ such that $I(\phi , \sigma_\rho) \neq \emptyset$.
Moreover,
\begin{enumerate}
\item[(i)]
if $\Gamma$ is a horseshoe of a continuous flow $f$ then $I(f\vert _{\Lambda }) \neq \emptyset$;
\item[(ii)]
if $\sigma:\Sigma\to \Sigma$ is topologically conjugate by a
homeomorphism $\Pi:\Lambda\to \Sigma$ to the restriction of  a
$C^{1+\alpha}$-diffeomorphism $T$ on a $u$-conformal elementary
horseshoe $\Lambda$,
% then
$$
   \dim_H I(\hat\phi,(T\vert _\Lambda )_{\hat\rho}) \geq 1 + \mathrm{d}_s(\Lambda) +\mathrm{d}^u(\Lambda),
$$
where $\hat\rho=\rho\circ \Pi$ and $\hat{\phi}=\phi \circ \hat\Pi$ with $\hat\Pi : \Lambda_{\hat\rho} \to \Sigma_\rho$ given by $\hat\Pi(x,t)=(\Pi(x),t)$.
\end{enumerate}
\end{lem}
%In particular, $I(\sigma _\rho ) \neq \emptyset$ due to \eqref{eq:0106b}.
%In particular,  if $\Lambda$ is a horseshoe of a continuous flow $f$ then $I(f\vert _{\Lambda }) \neq \emptyset$ due to \eqref{eq:0209c} and \eqref{eq:0106b}.
%\marginpar{\tiny Proof in end of article}

\subsection{Topological entropy}
%As before, the proof of full Hausdorff dimension of the irregular
%set of geometric Lorenz flows is more delicate than the proof of
%full entropy, so we start from topological entropy.
We first show that the irregular set of any geometric Lorenz flow has full topological entropy.
\begin{rem}\label{rem:0219}
Pacifico and Sanhueza recently showed that for any continuous flow $f$ on a compact metric space with (almost) specification property, the irregular set $I(f)$ of $f$ has full topological entropy (\cite[Theorem 5.8]{PS19}).
However, it is known that any geometric Lorenz flow does not satisfy the specification property (\cite{SVY15}).
Furthermore, Thompson  showed in \cite{T10} that for any suspension flow $f$ of a homeomorphism with the specification property, the irregular set $I(f)$  has full topological  entropy ($f$ itself is not required to satisfy the specification property).
However, again, it is unclear that the Poincar\'e map of a geometric Lorenz flow satisfies the specification property in general.
%when it does not
%hold that the slope of the Poincar\'e map along the unstable
%direction ($F$ in \eqref{eq:0905}) is greater than $\sqrt 2$
%everywhere.
%on a compact metric space with (almost) specification property, the irregular set $I(f)$ of $f$ has full entropy (\cite[Theorem 5.8]{PS2019}).
% so we will approximate  the geometric Lorenz flow  by horseshoes.
\end{rem}

\begin{proof}[Proof of Proposition \ref{prop:0904} $\mathrm{(i)}$]
Let $V$ be a $C^1$ vector field which defines a geometric Lorenz attractor $\Lambda$.
%Since any vector field $W$ in a neighborhood of $V$ also defines a geometric Lorenz attractor, it suffices to
We first show that the irregular set of $f_V$ has full topological entropy.
If the topological entropy of $f_V$ is zero, then we have nothing to prove.
So, we  assume that $h_{\mathrm{top}}(f_V)>0$.
%As mentioned previously,
Recall that for any $C^1$-close vector field $W$ to $V$, $\Lambda
_W$ is a singular-hyperbolic attractor with the trapping region
$D$. In particular, all singularities of $W$ in $D$ are
hyperbolic. This implies that
%It is known that the geometric Lorenz flow
 $f_V$ satisfies the \emph{star property}, i.e.~there is a
 neighborhood $\mathcal N \subset \mathcal X^1(D)$ of $V$ such that for any $W\in \mathcal N$, all singularities and all periodic orbits of $f_{W}$ are hyperbolic.
 %  (cf.~the remark after Theorem A of \cite{LSWW20}).
Hence, we can apply the entropy approximation theorem   of star
flows by horseshoes~\cite[Proposition 1.1]{LSWW20} to $f_W$ for
any $W\in\mathcal{N}$: for every $\epsilon >0$, there exists a
horseshoe $\Gamma _\epsilon $ of $f_W$ such that
\begin{equation}\label{eq:0901}
h_{\mathrm{top}} (\Gamma _\epsilon, f_W) > h_{\mathrm{top}}
{(\Lambda_W , f_W)} -\epsilon = h_{\mathrm{top}} (f_W) -\epsilon.
\end{equation}
Let $\sigma _{\rho _\epsilon}$ be the suspension flow of the
left-shift operator $\sigma $ on the topologically mixing shift of
finite type $\Sigma^\epsilon$ over a continuous function $\rho
_\epsilon$ which is conjugate to $f_W$ on $\Gamma _\epsilon$.
%by $\pi_\epsilon :\Sigma ^\epsilon  _{\rho _\epsilon }\to \Gamma_\epsilon $.
Then, since $I(\sigma _{\rho_\epsilon} ) \neq \emptyset$ by
Lemma~\ref{lem:0904} and a subshift of finite type is topologically mixing if and only if it has the periodic specification property (cf.~\cite{kwietniak2016panorama}),
it follows from the previously-mentioned
Thompson's theorem (\cite[Theorem 5.1 and Lemma 5.4]{T10}) that
\[
h_{\mathrm{top}} (I(\sigma _{\rho _\epsilon} ) , \sigma _{\rho
_\epsilon} ) = h_{\mathrm{top}} ( \sigma _{\rho _\epsilon}
)=h_{\mathrm{top}} ( \Gamma _\epsilon , f_W ).
\]
On the other hand,
\begin{equation}\label{eq:0210c}
h_{\mathrm{top}} (I(\sigma _{\rho _\epsilon} ) ,\sigma_{\rho
_\epsilon} ) \leq h_{\mathrm{top}} ( I(f_W) \cap \Gamma _\epsilon
, f_W )\leq h_{\mathrm{top}} (I(f_W)   , f_W )
\end{equation}
(the first inequality will be proven in a slightly more general form, see Remark  \ref{rem:0210d}).
%\marginpar{\tiny ã?¡ã‚‡ã?£ã?¨é›‘é?Žã?Žã‚‹ã?‹ï¼Ÿ}
Combining these estimates with \eqref{eq:0901}, we have
\[
h_{\mathrm{top}} (f_W) -\epsilon < h_{\mathrm{top}} (I(f_W)  ,
f_W).
\]
Since $\epsilon $ is arbitrary, we conclude that the irregular set
of $f_W$ has full topological entropy.
%Furthermore, since  any
%$C^1$ vector field $W$ which is sufficiently
%  $C^1$-close to $V$ is obviously star, we immediately get the conclusion.
\end{proof}

\subsection{Hausdorff dimension}\label{ss:LH}
Next, we will show that  the irregular set of any geometric Lorenz flow has full Hausdorff dimension.
\begin{proof}[The proof of Proposition \ref{prop:0904} $\mathrm{(ii)}$]
Recall  \eqref{eq:0905} for   the Poincar\'e map $P$ of $f_V$. We   identify $P$ with the two-dimensional map $ (x,y)\mapsto (F(x), G(x,y))$.
%\marginpar{\tiny å°Žå…¥ä¸?è¦?ï¼Ÿ}
%Let $V$ be a $C^1$ vector field which defines a geometric Lorenz attractor $\Lambda$.
%For clarity, we employ the construction of geometric Lorenz flows of \cite{MPR20}.
%In particular, in this setting,
%\[
%\text{$V(x,y,z) =(\lambda _1x,\lambda _2 y,\lambda _3z)$ on $[-1, 1]^3$ with $0<-\lambda _3<\lambda _1<-\lambda _2$},
%\]
%and the Poincar\'e map $T: [-1,1]^2 \times \{1\} \setminus W^s(0,0,0) \to [-1,1]^2 \times \{1\}$ of $f_V$ is well-defined and of the skew-product form
% \begin{equation}\label{eq:0905}
% \text{$T(x,y,1)=(F(x),G(x,y) ,1)$ for $x\in[-1,1]\setminus \{0\}$ and $y\in [-1,1]$, }
% \end{equation}
% where $F:[-1,1]\setminus \{0\} \to [-1,1]$ is a piecewise expanding map and $y\mapsto G(x,y)$ is a contraction map for every $x\in [-1,1]\setminus \{0\}$.
%Furthermore,  one can take an open set $D$ including $[-1,1]^3 $ such that $\Lambda = \bigcap _{t\geq 0}f^t_V(D)$.
%See  Section 2 of \cite{MPR20} for more precise construction.
%
By~\cite[Thm.~1]{MPR20}, there is an increasing sequence of
regular Cantor sets $(C_k)_{k\in \mathbb N}$ for $F$ such that
$\dim_{\mathrm H} C_k\rightarrow 1$ as $k\to \infty$. Fix $k\geq
1$. We consider
\[
\Lambda_{k} = \{(x,y) \in [-1, 1]^2 \mid (x,y,1)\in
\Lambda_V ,\; x\in C_k \}.
\]
Since $C_k$ is a regular Cantor set, $\Lambda_{k}$ is an 
%\color{red} locally maximal hyperbolic \color{black}
elementary horseshoe of $P$.
%Actually,
%\color{red} restricting to
%the topologically mixing pieces \color{black} of $\Lambda_k$ if
%necessary \color{red} (recall Remark \ref{rem:0410}),
%\color{black} we can assume that $\Lambda _k$ is topologically
%mixing.
%\color{red}
%Hence, there exists $E\subset \Lambda_k$ (which a part of the
%boundary of a Markov partition of $\Lambda _k$) such that
%$\mathrm{dim}_{\mathrm{H}} E< \mathrm{dim}_{\mathrm{H}} \Lambda_k$
%and $P|_{(\Lambda_k \setminus E)}$
%is conjugate to a topologically mixing subshift of finite type.
%which is a transitive hyperbolic invariant set
%{for $P$. Actually, since $C_k$ is regular Cantor
% set for base map $F$ of the skew-product $P$ with contracting
% fibers we have that $\Lambda_k$
%far from the singularity $0$ by construction.
%Actually, since
%$C_k$ is a regular Cantor set
%\marginpar{\tiny citation?}
Hence,  the restriction $P \vert _{\Lambda _k}$ of $P$ on
$\Lambda_k$ is topologically conjugate to the shift operation
$\sigma: \Sigma _k \to \Sigma _k$ on a topologically mixing
subshift of finite type  $\Sigma _k$ with a finite alphabet.
%(cf.~the proof of Theorem 4.1 of \cite{BS00}).
% together with  \cite[Lemma 8.4]{BS00}).
%Since $\mathrm{dim}_{\mathrm{H}} E< \mathrm{dim}_{\mathrm{H}} \Lambda_k$, we can assume that $E=\emptyset$ without loss of generality. \color{black}
That is, there is a homeomorphism $\Pi :\Lambda_k \to \Sigma _k$ such that $ \sigma \circ \Pi = \Pi \circ P$.
Let $\hat\rho(x,y)$ be the first return time of $(x,y,1)$ for $(x,y)\in \Lambda _k$  to the Poincar\'e section $S$ by $f_V$.
Since $C_k$ is far from the singularity $0$, it is straightforward to see   that $ \hat\rho : \Lambda _k\to \mathbb R_+$ is uniformly bounded and continuous.
Set $\rho=\hat{\rho}\circ \Pi^{-1}$.
From Lemma~\ref{lem:0904}, we get that there exists a continuous function $\phi: (\Sigma_k)_\rho \to \mathbb{R}$ such that
$$
   \dim_{\mathrm{H}} I(\hat\phi,(P|_{\Lambda_k})_{\hat\rho})\geq 1+ \mathrm{d}_s(\Lambda_k) +   \mathrm{d}^u(\Lambda_k),
$$
where  $\hat{\phi}=\phi \circ \hat\Pi$ with $\hat\Pi : (\Lambda_k)_{\hat\rho} \to (\Sigma_k)_\rho$ given by $\hat\Pi(x,y,t)=(\Pi(x,y),t)$.
Moreover,
$$
\mathrm{d}_s(\Lambda_k)=1 \ \ \text{and} \ \ \mathrm{d}^u(\Lambda
_{k}) =\dim_{\mathrm H}(W^u _{loc} (x) \cap \Lambda _k) =
\dim_{\mathrm H} C_k \ \ \text{for each $x\in \Lambda _k$}.
$$

Consider now
\[
\Gamma _k = \left\{ f^t _V(x,y,1) \mid (x,y)\in \Lambda _k, \;
0\leq t< \hat\rho(x,y) \right\}.
\]
%\marginpar{\tiny ã?“ã‚Œã? ã?¨forward invariantã?«ã?—ã?‹ã?ªã‚‰ã?ªã?„ã?‹ã?®ã?§åŽ³å¯†ã?«è¨€ã?ˆã?°natural extensionï¼Ÿã‚’è€ƒã?ˆã‚‹å¿…è¦?ï¼Ÿã?Œæµ?çŸ³ã?«çœ?ç•¥? suspension semiflowã?«ã?¤ã?„ã?¦è¨€ã?ˆã?°ã?„ã?„ã? ã?‘ã?¨è¨€ã?ˆã?°æ¥½ï¼Ÿ}
%Furthermore, let   $  \rho (x,y) $ be the length of the arc $\{ f^t_V(x,y,1) \mid 0\leq t\leq \tau (x,y)\}$ for each $(x,y)\in \Lambda _k$.
%  and $\widetilde f =\{ \widetilde f^t\} _{t\in \mathbb R}$ the suspension  flow of $P$  over $\widetilde \rho$ on the $\widetilde \rho$-suspension space $(\Lambda _{k})_{  \widetilde \rho} $.
%Then,
%by using
%due to the fact that
Notice that $f \vert _{\Gamma _k}$ is topologically conjugate to the suspension flow $(P\vert _{\Lambda _k})_{\rho}$ of $P \vert _{\Lambda _k}$ over $\rho$ by the map $\pi : (\Lambda _k)_{\rho} \to  \Gamma _k $, where $\pi(x,y,t) =f^t_V(x,y,1)$.
Observe that from the smooth dependence of the flow with respect to the initial conditions,   since $V$ is a $C^1$ vector field, $\pi$ is also of  class $C^1$ (cf.~\cite[Appendix~B]{duistermaat2000lie}).
 In particular, since
\[
I (  \hat \phi\circ \pi,f_V \vert _{\Gamma _k}) = \pi\left(I (\hat\phi , (P\vert _{\Lambda _k})_{ \hat\rho })\right),
\]
it holds that $ \dim_{\mathrm{H}} I (  \hat \phi\circ \pi,f_V \vert _{\Gamma _k}) = \dim_{\mathrm{H}} I (\phi , (P\vert _{\Lambda _k})_{ \hat\rho }) $.
So, we get
\[
\dim_{\mathrm H} I( f_V) \geq \dim_{\mathrm H} I( f_V\vert _{\Gamma _k })\geq \dim_{\mathrm H} I( \hat \phi \circ  \pi ,f_V\vert _{\Gamma _k }) \geq  2+  \dim_{\mathrm H} C_k  \to 3
\]
as $k\to \infty$.
That is, the irregular set of $f_V$ has full Hausdorff dimension.

Now, we argue that the same is true for any vector field $W$ $C^1$-close to $V$ under the assumption $- \frac{\lambda _2}{\lambda _1} > -\frac{\lambda _3}{\lambda _1} + 2$.
Recall that there exists a $C^1$-neighborhood $\mathcal N\subset \mathcal X^1(D)$ of $V$ such that for each $W\in \mathcal N$, the maximal invariant set $\Lambda _W= \bigcap _{t\geq 0} f^t_W(D)$ is  singular-hyperbolic.
Moreover, it is known that the associated Poincar\'e map $P_W$ preserves the strong stable foliation $\mathcal F_W$ with $C^1$ leaves and  the holonomies along the leaves are of class $C^1$ (cf.~\cite[Section 2.2]{MPR20}).
On the one hand,
it follows from  Proposition 1 and Corollary D of  \cite{MPR20} that, by taking  $\mathcal N$ small if necessary, for each $W\in \mathcal N$,
the quotient map $F_W : S_0 / \mathcal F_W \to S/ \mathcal F_W$ associated with the  Poincar\'e return map $P_W: S_0\to S$
% whose quotient along $\mathcal F_W$ is a one-dimensional Lorenz map whose slope is strictly larger than $\sqrt 2$ with a nice control of the derivative (i.e.~(f1)-(f3) of \cite{MPR20}).
%On the other hand, it follows from \cite[Corollary D]{MPR20} that if a one-dimensional Lorenz map satisfying (f1)-(f3) is a $C^2$ map, then
%there is
has a sequence of Cantor sets  $( C_k^W) _{k\in \mathbb N}$ with
$\dim_{\mathrm H}(C^W_k) \to 1$, provided that $F_W$ is a
$C^2$ map.
%Combining these arguments,
On the other hand, if it satisfies  that $- \frac{\lambda
_2}{\lambda _1} > -\frac{\lambda _3}{\lambda _1} +r$ with some $r\geq 2$, then
$\mathcal F_W$ is a $C^r$ smooth foliation and the quotient map is
also of class $C^r$ (see Section 2.2 of
\cite{MPR20} for details).
%In particular, under our hypothesis, the holonomies are of class $C^2$.
Therefore, under our hypothesis, the quotient map $F_W$ is a $C^2$ map and  thus there is a sequence of Cantor sets $( C_k^W) _{k\in \mathbb N}$ of $F_W$ for each $W\in \mathcal N$.
Since  the Poincar\'e map $P_W$ preserves the strong stable foliation $\mathcal F_W$, we can get the conclusion by repeating the argument for the proof of full Hausdorff dimension of the irregular set of $f_V$.
%By these facts, the proof for the flow $f_W$ generated by the perturbed vector filed $W$ is similar to that of $f_V$.
\end{proof}

Finally, we ask two questions about the generalization of
Theorem~\ref{prop:0904}.

%\marginpar{\tiny OK?}
\begin{question}
Does there exist an open and dense set $\mathcal U$ of the space of three-dimensional smooth vector fields defining a singular-hyperbolic attractor  such that the irregular set of $f_V$ has full topological entropy and full Hausdorff dimension for every $V\in \mathcal U$?
\end{question}

\begin{question}
 If $f_V$ is the geometric Lorenz flow, then does the completely irregular set $CI(f)$ (cf.~\cite{T17}) have full topological entropy and full Hausdorff dimension?
Moreover, does this hold  for $f_W$ for any vector field $W$ $C^1$-close to $V$?
\end{question}

\section{The proofs}\label{s:proof}

\subsection{Proof of Theorem \ref{thm-entropy-one-dimensional}}\label{S:5.1}

%The proof in~\cite{NY20} is once
%again based in Proposition~\ref{prop:1}. Using Hofbauer's Markov
%diagram extension of $f$, the authors in~\cite{NY20} find a
%sequence of $f$-invariant compact sets having specification
%property and whose entropy converges to the topological entropy of
%$f$. On the other hand, according to~\cite{Block}, the set of
%periodic measures of a continuous interval map with positive
%topological entropy is dense in the set of its ergodic invariant
%measures. Then, above result implies, in this case, that the set
%of irregular points of a piecewise monotonic continuous map of a
%compact interval has full topological entropy.
%
%{proofs}
%
%\begin{proof}
We start from the following preliminary lemma.
\begin{lem}\label{lem:0210b}
Let $T: X'\to X'$ and $S:Y\to Y$ be continuous maps on metric
spaces $X'$ and $Y$. Assume that there exist a closed invariant
subset $X\subset X'$, invariant subsets $X_0 \subset X$ and $Y_0
\subset Y$, and a continuous map  $\varphi :X \to Y$ such that
$\varphi :X _0 \to Y_0$ is a {surjection} and $\varphi \circ T =S
\circ \varphi$ on $X_0$ (i.e.~$T: X_0\to X_0$ and $S: Y_0\to Y_0$
are topologically semi-conjugate by $\varphi$). Then, it holds
that
%the following inclusion holds:
\[
\varphi ^{-1}(Y_0 \cap I(S)) \subset X_0 \cap I(T).
\]
\end{lem}
\begin{proof}
Let $x\in\varphi ^{-1}(Y_0 \cap I(S))$.
Then $x\in \varphi ^{-1}(Y_0) =X_0$ and  $\varphi(x)\in Y_0 \cap I(S)$.
Thus, there exists a continuous map  $\phi: Y \to \mathbb{R}$  such that
$$
\lim _{n\to\infty} \frac{1}{n}\sum_{j=0}^{n-1} \phi (S^j(\varphi(x))
$$
does not exist.
Let $\psi : X\to \mathbb{R}$ be a continuous function given by $\psi = \phi\circ \varphi$, and extend it to a continuous function on $X'$ by the Tietze extension theorem.
Then,
$$
\lim _{n\to\infty} \frac{1}{n}\sum_{j=0}^{n-1} \psi (T^{j}(x)) =
\lim _{n\to\infty} \frac{1}{n}\sum_{j=0}^{n-1} \phi \circ \varphi (T^{j}(x)) =
\lim _{n\to\infty} \frac{1}{n}\sum_{j=0}^{n-1} \phi (S^j(\varphi(x))
$$
also does not  exist.
Hence, $x \in X_0 \cap I(T)$.
\end{proof}
\begin{rem}\label{rem:0210d}
By a similar argument, we can show the flow version of Lemma \ref{lem:0210b}, that is, Lemma \ref{lem:0210b} with continuous flows $T=\{T^t\}_{t\in\mathbb R}$ and $S=\{S^t\}_{t\in\mathbb R}$ in place of continuous maps $T$ and $S$.
Note that \eqref{eq:0210c} immediately follows by applying it to $X'=X=D$, {$X_0=\Gamma _\epsilon$}, $T=f_V$, $Y=Y_0 = \Sigma _{\rho _\epsilon}^\epsilon$, $S=\sigma _{\rho _\epsilon}$ and $\varphi =\pi ^{-1}$.
\end{rem}
Now we prove Theorem \ref{thm-entropy-one-dimensional}.
%\begin{proof}[Proof of Theorem \ref{thm-entropy-one-dimensional}]
If the topological entropy of $f$ is zero, then we have nothing to prove.
Thus, we can assume that $h_{\mathrm{top}}(f)>0$.
According to Misiurewicz's theorem (see~\cite[Theorem 4.7]{Rue17}) for every $0<\lambda <h_{\mathrm{top}}(f)$,
%and every $N\in \mathbb{N}$,
there exist intervals $J_1,\ldots,J_p$ and an integer  $k$ %$n\geq N$
such that $(J_1,\dots,J_p)$ is a strict $p$-horseshoe of $f^k$ and
$$
  \frac{\log p}{k}\geq \lambda.
$$
A \emph{strict $p$-horseshoe} of $f$ is a collection of $p$ pairwise disjoints intervals $(J_1,\dots,J_p)$ such that $J_1\cup \dots \cup J_p \subset f(J_i)$ for all $i=1,\dots,p$.
Moreover, according to Theorem 5.15 and the remarks following Theorem 5.8 of \cite{Rue17}, there exist an  $f^k$-invariant Cantor set $X\subset J_1\cup \dots \cup J_p$ and a continuous map $\varphi: X \to \Sigma$ with $\Sigma= \{1,\dots,p\}^{\mathbb{N}}$ such that
\begin{enumerate}
\item $\varphi$ is a semi-conjugacy between $f^k|_X$ and the shift map $\sigma: \Sigma \to \Sigma$;
%\item $f|_X$ is transitive;
\item there exists an $f^k$-invariant countable set $E \subset X$ such that $\varphi$ is one-to-one on $X\setminus E$ and two-to-one on $E$.
\end{enumerate}
In particular, since $\sigma:\Sigma \to \Sigma$ is a factor of $f^k|_X$, we have
\[
k\cdot h_{\mathrm{top}}(X,f) = h_{\mathrm{top}}(X,f^k) {\geq}
h_{\mathrm{top}}(\sigma)=\log p,
\]
 and consequently, we get that
\begin{equation} \label{eq:des}
  h_{\mathrm{top}}(f)\geq h_{\mathrm{top}}(X,f)\geq \frac{\log p}{k} \geq \lambda.
\end{equation}
That is, the topological entropy of $f$ is approximated by the topological entropy of $X$ (a horseshoe).
However, we do not know if $X$ has specification property for $f^k$ (notice that $f^k|_X$ is an extension of the full shift  but the specification property is preserved a priori only  by factors).
%at least that the semi-conjugacy was actually a conjugacy).
%\marginpar{\tiny argument in Section ??ã?®æ–¹ã?Œbetter?}
Thus, we cannot directly apply the argument in Subsection \ref{S:2.1} of approximation by horseshoes with specification.
However, since the semi-conjugacy is, in fact, almost a conjugacy (that is, $\varphi$ is  a one-to-one map  except a countable set), we can still approximate the topological  entropy of $f$ by the topological  entropy of the irregular set contained in the horseshoe as follows:

\begin{lem}\label{c:0901}
$h_{\mathrm{top}}(X\cap I(f),f) \geq  \frac{\log p}{k}$.
\end{lem}
\begin{proof}
Let $X_0=X\setminus E$ and $\Sigma_0=\Sigma\setminus A$, where $A=\varphi(E)$.
Notice that $A$ is also a countable set and that $I(\sigma)\setminus A= \Sigma_0 \cap I(\sigma)$ and $X_0\cap I(f)=X \cap I(f)\setminus E$.
Observe that since the topological entropy of the countable union of sets  equals to the supremum of the topological entropy of each set
% $h_{\mathrm{top}}(\bigcup _{i\in \mathbb N} B_i,f) =\sup _{i\in \mathbb N}h_{\mathrm{top}}(B_i,f)$ for any $B_1, B_2, \ldots$
%Bowen-Hausdorff topological entropy has countable-stability
(see \cite[Proposition 2(c)]{Bow73}), the topological entropy of a countable set is zero.
%\marginpar{\tiny closed under countable intersection?ã‚‚ã?—ã?ã?¯outer measureã?ªã?®ã?§$\sigma$-subadditive?}
We also have that $\varphi:X_0\to \Sigma_0$ is a continuous bijection.
By virtue of (5.9) in  \cite[Theorem 5.15]{Rue17}, it is not difficult to see that $\varphi: X_0\to \Sigma_0$ is a homeomorphism.
Thus, $f^k:X_0\to X_0$ and $\sigma:\Sigma_0\to \Sigma_0$ are conjugate.
Since topological entropy is  invariant under conjugacy (see~\cite[Proposition 2(a)]{Bow73}), we get that $h_{\mathrm{top}}(Z,\sigma)=h_{\mathrm{top}}(\varphi^{-1}(Z),f^k)$ for all $Z\subset
\Sigma_0$.
Hence, according to~\cite[Theorem B]{FFW01}~(see also \cite{BS00}) we get that
\begin{align*}
\log p &= h_{\mathrm{top}}(\sigma)=h_{\mathrm{top}}(I(\sigma),\sigma)
\\
&= h_{\mathrm{top}}(I(\sigma)\setminus A,\sigma)=h_{\mathrm{top}}(\Sigma_0 \cap
I(\sigma),\sigma) =
h_{\mathrm{top}}(\varphi^{-1}(\Sigma_0 \cap I(\sigma)),f^k)  \\
&\leq h_{\mathrm{top}}(X_0 \cap I(f),f^k)  = h_{\mathrm{top}}(X\cap I(f)\setminus
E,f^k) \\ &=h_{\mathrm{top}}(X\cap I(f),f^k)=k \cdot h_{\mathrm{top}}(X\cap I(f),f).
\end{align*}
The inequality above follows from the inclusion $
\varphi^{-1}(\Sigma_0\cap I(\sigma))\subset X_0 \cap I(f)$, which
is a straightforward consequence of Lemma \ref{lem:0210b}.
%To show this, given $x\in \varphi^{-1}(\Sigma_0\cap I(\sigma))$ we have that $x\in X_0$ and $\varphi(x)\in \Sigma_0 \cap I(\sigma)$ and thus there exists a continuous map  $\phi: \Sigma \to \mathbb{R}$  such that the sequence of Birkhoff average
%$$
%\frac{1}{n}\sum_{i=0}^{n-1} \phi (\sigma^i(\varphi(x))
%$$
%does not converges.
%Consider now the continuous observable $\psi=\phi\circ \varphi: X\to \mathbb{R}$.
%Then we have that  the sequence of Birkhoff averages
%$$
%\frac{1}{n}\sum_{i=0}^{n-1} \psi (f^{ki}(x)) =
%\frac{1}{n}\sum_{i=0}^{n-1} \phi \circ \varphi (f^{ki}(x)) =
%\frac{1}{n}\sum_{i=0}^{n-1} \phi (\sigma^i(\varphi(x))
%$$
%also does not  converge.
%Thus $x \in X_0 \cap I(f^k)\subset X_0\cap I(f)$.
\end{proof}
Finally, Lemma \ref{c:0901} and \eqref{eq:des} imply that
\[
h_{\mathrm{top}}(f)\geq h_{\mathrm{top}}(I(f),f) \geq h_{\mathrm{top}}(X\cap I(f),f)\geq \lambda.
\]
Since $\lambda <h_{\mathrm{top}}(f)$ is arbitrary, this immediately completes the proof of Theorem \ref{thm-entropy-one-dimensional}.

\subsection{Proof of Theorem~\ref{Contract}}\label{Contractproof}
Let $\mathcal{U}$ be the $C^1$-neighborhood of $f$ given in~\cite{Yang18} so that any $C^{1+}$ diffeomorphism in
$\mathcal{U}$ has mostly contracting center.

\begin{lem} Given $\epsilon >0$ there exists a
dense set $\mathcal{D}_\epsilon \subset \mathcal{U}$ such that for
$g\in \mathcal{D}_\epsilon$ there is an elementary horseshoe
$\Lambda_g$ with $\mathrm{d}^u(\Lambda_g)> 1-\epsilon$.
\end{lem}
\begin{proof}
Since $C^r$-maps are dense in $\mathcal{U}$, any element of
$\mathcal{U}$ can be $C^1$-approximated by $\tilde{f}$ which is
$C^2$ and so has mostly contracting center. Consider $\mu$ an
ergodic Gibbs u-state for $\tilde{f}$. Since $\dim E^u=1$ and
$\tilde{f}$ is mostly contracting, $\mu$ can only have one
positive Lyapunov exponent, which we denote by $\lambda^u_\mu$. As
$\mu$ is a Gibbs u-state, then the Pesin formula is satisfied and
$h_\mu(\tilde{f})=\lambda^u_\mu$.

Because $\mu$ is a hyperbolic measure, by the theorem of
Katok~\cite{KH95} we can approximate in entropy the measure by
elementary horseshoes. One of the results of the
recent work of~\cite{ACW20} states that this horseshoe can be
linearized in $C^1$-topology. More specifically, applying
\cite[Theorem B']{ACW20}, we can $C^1$-approximate $\tilde{f}$ by
a $C^2$ map $g$, which has in certain coordinates an affine linear
horseshoe $\Lambda_{g}$. That is, in a neighborhood of each point
in $\Lambda_{g}$ there exists a change of coordinates so that the
map $g$ coincides with an affine linear transformation. The
constant linear part of the transformation is given by a diagonal
matrix $A$ whose entries are independent of the point in
$\Lambda_{g}$.
%afffine linear transformation with the constant linear part given by a diagonal matrix $A$.
Moreover, the logarithms of the eigenvalues of $A$ are arbitrary close to the exponents of $\mu$ and the topological entropy of $\Lambda_{g}$ is close to the metric entropy $h_\mu(g)$.
Since $\mu$ has only one positive Lyapunov exponent, thus $A$ has a unique unstable eigenvalue which we denote by $A^u$ with $\lambda^{u}_A=\log(A^u)$.
%Denote by $\lambda^{u}_A$ the logarithm of the unique unstable eigenvalue of $A$, which in turn will be denoted by $A^u$.

Actually, we will take the horseshoe to be the so-called
\textit{standard affine horseshoe}, $\Lambda'_{g}\subset
\Lambda_{g}$ (see \cite[Definition 7.4 and Proposition
7.8]{ACW20}). In this definition, the horseshoe satisfies some
extra properties based on the classical model. We observe that the
standard affine horseshoe of \cite[Proposition 7.8]{ACW20} is with
respect to some iterate ${g}^N$ of the map $g$ satisfying
$\frac{1}{N} h_{\mathrm{top}}(\Lambda'_g ,g^N)\geq
h_{\mathrm{top}}(\Lambda_g,g)- \epsilon$ for  some arbitrary small
$\epsilon$. Observe that the Lyapunov exponents of $\Lambda'_g$
with respect to $g^N$ then will also be multiplied by $N$. In the
arguments below, without loss of generality, we may assume that
$\Lambda_g$ is the \textit{standard affine horseshoe} for $g$. The
horseshoe $\Lambda_g$ is defined by  affine diagonal maps, and
projecting these maps to the line spanned by the unstable
direction, one obtains an iterated function system on the line
given by affine maps of the form: $T_i(x)=A^u x+c_i$. By taking
the inverses $T_i^{-1}$ we may assume we are working with affine
contractions in the line having an invariant attractor
$\Lambda^u_g$. With respect to the horseshoe $\Lambda_g$, it is a
standard affine horseshoe, and in this case the entropy
$h_{\mathrm{top}}(\Lambda_g,g)$ is related to the number of
expanding ``legs'' of the horseshoe, which in turn is equal to the
number of the affine maps in the iterated system. In particular,
the entropy $h_{\mathrm{top}}(\Lambda_g,g)$ is equal to $\log(k)$,
where $k$ is the number of maps of the iterated system $\{T_i\}$.

To estimate the Hausdorff dimension of the invariant set
$\Lambda^u_g$ we will use the classical result of
Falconer~\cite{Falconer} for affine contractions on the line
satisfying the open set condition. The open set condition has to
do with disjointness of the images of $T_i$ and since $\Lambda_g$
is a standard affine horseshoe, the iterated system $T^{-1}_i$
will satisfy this property (\cite[Definition 7.4]{ACW20}). To
calculate $\dim_{\mathrm H}(\Lambda^u_g)$, one has to  resolve for
$s$ the equality $k\cdot (A^{u})^{-s}=1$. Then we obtain the
classical formula
$$s=\dim_{\mathrm H}(\Lambda^u_g)= \log k \cdot \log(A^u)^{-1}=  h_{\mathrm{top}}(\Lambda_g,g)\cdot(\lambda^u_A)^{-1}.$$
By construction, $h_{\mathrm{top}}(\Lambda_g)$  is
arbitrary close to
 $h_\mu(\tilde{f})$ and $\lambda^u_A$ is close to
the unique unstable exponent $\lambda^u_\mu$ of $\mu$. Since
$h_\mu(\tilde{f})=\lambda^u_\mu$, then
$h_{\mathrm{top}}(\Lambda_g,g)$ is arbitrarily
close to $\lambda_A^u$ and consequently $\dim_{\mathrm
H}(\Lambda^u_g)$ is arbitrary close to 1. Moreover, observing that
$\dim_{\mathrm H}(\Lambda^u_g)=\mathrm{d}^u(\Lambda_g)$ we
conclude the proof of the lemma.
%To summarize, we have proven the following.
\end{proof}
Let $g$ be as in the previous lemma. Consider
$\tilde{g}$ that is $C^1$-close to g and its respective horseshoe
$\Lambda_{\tilde{g}}$, which is the continuation of $\Lambda_g$.
There exists a H\"{o}lder homeomorphism $\phi$ which conjugates
$\Lambda_g$ with $\Lambda_{\tilde{g}}$. Moreover since the
unstable dimension is one, for every $x\in \Lambda_g$, the
H\"{o}lder constant of $\phi\vert_{\Lambda_g\cap W^u(x)}$ is
arbitrary close to 1 (see~\cite[ch.19]{KH95}). In particular,
$\mathrm{d}^u(\Lambda_g)$ varies continuously with respect to the
continuation of the horseshoe $\Lambda_g$ in a $C^1$-neighborhood
of $g$. Thus, we actually have the following.

\begin{lem} Given $\epsilon >0$ there exists an open and dense set
$R_\epsilon \subset \mathcal{U}$ such that for $g\in R_\epsilon$
there exists an elementary horseshoe $\Lambda_g$ whose $u$-index
(dimension of the unstable bundle) is $1$ and with
$\mathrm{d}^u(\Lambda_g)> 1-\epsilon$.
\end{lem}
Then taking $\epsilon_n=1/n$ and $\mathcal{R}= \bigcap
R_{\epsilon_n}$, we obtain a residual set $\mathcal{R}$ in
$\mathcal{U}$ so that for $g\in\mathcal{R}$
$$\sup \big\{\mathrm{d}^u(\Lambda_g): \  \Lambda_g \ \text{is an elementary horseshoe with $u$-index $1$} \big\}=1.$$
Consider now a map $g\in\mathcal{R}$ and any
elementary horseshoe $\Lambda_g$
with $u$-index $1$. In particular,
$\mathrm{d}_s(\Lambda_g)=\dim(M )-1$. Applying
Proposition~\ref{Dim} we obtain that
\begin{align*} \dim_{\mathrm H}   I(g)  \geq \sup \big\{\mathrm{d}_s&(\Lambda_g)+\mathrm{d}^u(\Lambda_g): \\
&\Lambda_g \ \text{is an  elementary horseshoe with $u$-index $1$}
\big\}=\dim(M). \end{align*} This completes the proof of the
theorem.

\subsection{Proof of Theorem~\ref{thm:0219c}}\label{thm:0219c_proof}

We will need to explain some of the results coming from~\cite{BCF18}. Using the non-existence assumption of dominated
splitting  and the hypothesis that $\Delta^-(f,P)\geq
\Delta^+(f,P)$, we can $C^1$-approximate $f$ by a
$C^2$-diffeomorphism $g$ having a $u$-conformal affine linear
horseshoe $K_g$ with unstable index $d_u=\dim E^u_P$ and stable
index $d_s= \dim E^s_P$
%inside of the homoclinic class $H(P_g)$ of the continuation $P_g$ of $P$ for $g$
(see Theorem 4.1 and Proposition 4.6 of ~\cite{BCF18}). Moreover,
as explained in the proof of~\cite[Proposition 4.6]{BCF18} this
horseshoe has unstable dimension $\mathrm{d}^u(K_g)\geq d_u$.
Relating this with   Proposition~\ref{Dim} we get that
$\dim_{\mathrm H} I(g) \geq d_s+d_u=\dim M$.

Now consider a function $h$ that is $C^1$-close to g and its  respective horseshoe $K_h$, which is the continuation of $K_g$.
Then there exists a (unique) homeomorphism $\phi$, close to the identity, which conjugates $K_g$ with $K_h$.
The functions $\phi$ and $\phi^{-1}$ are actually H\"{o}lder continuous with the H\"{o}lder constant  arbitrarily close to $1$ as $h$ gets closer to $g$ (see again the proof of Proposition 4.6 of ~\cite{BCF18}).
Using this conjugation with respect to the unstable irregular set of $K_g$, and applying similar reasoning as in Proposition~\ref{Dim} we obtain the following.
Given $\epsilon>0$ there exists a neighborhood $\mathcal U_\epsilon$ of $g$ so that each $h\in \mathcal U_\epsilon$ satisfies that
%has a horseshoe $K_h$ with
$\dim_{\mathrm H} I(h)
%\geq \dim_{\mathrm H} (I(h) \cap K_h )
>\dim M-\epsilon$.

To conclude the $C^1$-genericity  in the  statement of the theorem, one can use known $C^1$-generic properties together with standard Baire arguments.
This is described in ~\cite{BCF18}, particularly in the proof of Theorem $5$.
%\end{proof}

\subsection{Proof of Lemma \ref{lem:0904}}\label{p:55}

%As a consequence of Proposition \ref{Dim}, one can get large
%Hausdorff dimension of the irregular set  for
%$u$-conformal dynamical systems which have hyperbolic sets with
%large unstable Hausdorff dimension.
We first
%prepare some auxiliary notations and
recall results of \cite{BS00} and \cite{T10} for preparation.
Let $f$ be a continuous map on a compact invariant set $\Lambda$ of a smooth manifold and   $\rho :   \Lambda \to \mathbb R_+$ a continuous function.
Given a continuous function $\phi : \Lambda \to \mathbb R_+$, we define a \emph{$\rho$-weighted $\phi$-irregular set $I_\rho (\phi , f) $} by
\begin{equation}\label{eq:1223}
I_\rho (\phi , f) =\left\{ x\in \Lambda :\,   \lim _{n\to \infty} \frac{S_n\phi (x)}{S_n\rho (x)} \; \text{does not exists}\right\},
\end{equation}
where
$$
S_n \psi (x) =\frac{1}{n} \sum_{j=0} ^{n-1} \psi \circ f^j
(x)
$$
for each real-valued continuous function $\psi$.
Feng, Lau and Wu showed that if $f$ is of class $C^{1+\alpha }$ and $  \Lambda$ is a repeller of $f$, then for each continuous function $\phi$ such that $I_\rho (\phi , f\vert _\Lambda ) \neq \emptyset$,
\[
\dim_{\mathrm H} I_\rho (\phi , f\vert _\Lambda ) = \dim_{\mathrm H}   \Lambda ,
\]
see \cite[Theorem 1.2]{FLW02} (compare it with \cite[Theorem~4.2 (1) and 7.1]{BS00}).
Therefore, the proof of Proposition~\ref{Dim}
%obviously
 works to prove that
\begin{equation}\label{eq:0219e}
I_\rho (\phi , f\vert _\Lambda ) \neq \emptyset \quad \Rightarrow
\quad \dim_{\mathrm H} I_\rho (\phi , f\vert _\Lambda ) \geq
\mathrm{d}_s(\Lambda) + \mathrm{d}^u(\Lambda)
\end{equation}
for each  $C^{1+\alpha}$-diffeomorphism $f$ of a compact manifold
with a $u$-conformal elementary horseshoe $\Lambda$ of $s$-index
$\mathrm{d}_s(\Lambda)$.
%This technical remark is useful  to analyze historic behavior for suspension flows, see Subsection \ref{ss:LH}.

Moreover, it follows from
\cite[Lemma 5.4]{T10} that for any continuous function $\phi : X
_\rho\to \mathbb R$, if we define a continuous function $\iota
(\phi ): X \to \mathbb R$ by
\[
\iota (\phi )(  x) = \int^{\rho (  x)}_0 \phi(  x,z)\, dz\quad
\text{for $  x\in X $},
\]
then we have
\begin{equation}\label{eq:0106b}
I (\phi,T_\rho) =\left\{ (  x,z) : \,   x\in I_\rho (\iota (\phi
),T), \; 0\leq z<\rho (  x)\right\}.
\end{equation}
%Recall \eqref{eq:1223} for
% the definition of $\rho$-weighted irregular set %$I_\rho (\psi,T)$
% $I_\rho (\iota (\phi ),T)$.
%%
In particular, $I (\phi,T_\rho)  \neq \emptyset$ if and only if
$I_\rho (\iota (\phi ),T) \neq \emptyset$.

\begin{proof}[Proof of Lemma \ref{lem:0904}]
%\subsubsection{Existence of $\phi$ for which $I(\phi , \sigma _\rho) \neq \emptyset$}
We first prove the first assertion in Lemma \ref{lem:0904} (i.e.~the existence of a continuous function $\phi$ for which $I(\phi , \sigma _\rho ) \neq \emptyset$).
Indeed, we will construct a  continuous map $\phi : \Sigma _\rho \to \mathbb R$ such that
$I_\rho (\iota (\phi ), \sigma ) \neq \emptyset$.
Due to \eqref{eq:0106b}, this immediately implies that $I(\phi , \sigma _\rho) \neq \emptyset$.
%Since the later claim (for the irregular set of horseshoes) is an immediate consequence of the former claim together with  \eqref{eq:0106b},
%we will prove the former claim.
We  follow the argument in \cite[Section 4]{Takens08}.
%Let $\Sigma \subset \{ 1, 2,\ldots , N\}$ be a topologically mixing subshift of finite type.
For clarity, we let $\Sigma $ endowed with a standard metric $d_{\Sigma}$ given by
 \[
 d_{\Sigma }(x,y) = \sum _{m\in \mathbb Z} \frac{\vert x_m -y_m \vert }{\beta ^{\vert m \vert
 }} \quad \text{for  $x=(x_m)_{m\in\mathbb{Z}}$ and
 $y=(y_m)_{m\in\mathbb{Z}}$}
 \]
where $\beta>1$.

 Set $\mathcal L(\Sigma ) = \left\{ x \in \bigcup _{n\geq 1} \{1, 2, \ldots , N\} ^n \mid C(x) \neq \emptyset \right\}$, where $C(x) = \{ y\in \Sigma \mid (y_1,\ldots ,y_n) =x\}$ for $x\in \{1, 2, \ldots ,N\}^n$.
 Since $\Sigma$ is a  topologically mixing subshift of finite type with $N\geq 2$, there are two different periodic points $p^0$
 %=(p_m)_{m\in\mathbb{Z}}$
  and $p^1$,
  %=(q_m)_{m\in\mathbb{Z}}$,
   and $\sigma : \Sigma \to \Sigma$ satisfies the specification property, i.e.~there is an integer $L > 0$ such that for any $x, y \in \mathcal  L(\Sigma )$, one can find $z \in \mathcal  L(\Sigma )$ with $\vert  z \vert = L$ such that $xzy \in \mathcal L(\Sigma )$, where $|z|$ is the length of the word $z$ and $x zy$ is the concatenation of $x$,  $z$ and $y$.
  Let $L_j$  be the period of $p^j$ ($j=0,1$).
  %$p^2$, respectively.
  %such that $p^1_m =1$ and $p^2_m=2$ for all $m\in \mathbb Z$.
  Since  $\Sigma$ is compact and $\rho : \Sigma \to \mathbb R_+$ is continuous, there is a constant $C>1$ such that
\begin{equation}\label{eq:0911}
C^{-1} \leq \rho (x) \leq C \quad \text{ for every $x\in \Sigma$.}
\end{equation}
Take a strictly increasing sequence of positive integers $(n_j)_{j\in \mathbb N}$ such that
 \begin{equation}\label{eq:0911b}
\liminf _{m\to \infty} \frac{n_{2m-1}L_1}{N_{2m-1}}
> \frac{2}{3}, \quad
\liminf _{m\to \infty}  \frac{n_{2m}L_0}{N_{2m}}
>1- \frac{1}{3C^2}, \quad N_m =mL + \sum _{k=1}^m n_k L_{(k\; \mathrm{mod} \; 2)}.
 \end{equation}
% where $\kappa (k) =1$ if $k$ is an odd number and $=2$
%For instance, it suffices taking $(n_j)_{j\in \mathbb{N}}$ such that
% \[
% n_{2m-1}L_1 > 2(n_1L_1 + \cdots + n_{2m-2}L_0), \quad  n_{2m} L_0> (3C^2 -1) (n_1 L_1+ \cdots +
% n_{2m-1}L_1).
% \]
Let $A_j$ be a small neighborhood of the orbit of $p^j$  ($j=0,1$) such that the closures of $A_1$ and $A_2$ are disjoint.
 We then  let  $\hat{x} =
(\hat{x}_m)_{m\in\mathbb{Z}}$ be a point such that
\[
[\hat{x}]_{0}^\infty =
%(\cdots \hat{x}_{-2}\hat{x}_{-1})
z_1 [p^1]_0^{n_1L_1-1} z_2 [p^0]_0^{n_2L_0-1} z_3 [p^1]_0^{n_3L_1-1} z_4 [p^0]_0^{n_4L_0-1} z_5 [p^1]_0^{n_5L_1-1} z_6 [p^0]_0^{n_6L_0-1} \cdots ,
% (\cdots \hat{x}_{-2}\hat{x}_{-1} .\;  \underbrace{11\cdots
%1}_{n_1} \; \underbrace{22\cdots 2}_{n_2} \;\underbrace{11\cdots
%1}_{n_3} \; \underbrace{22\cdots 2}_{n_4} \;  \underbrace{11\cdots
%1}_{n_5} \; \underbrace{22\cdots 2}_{n_6}  \cdots ).
\]
with some words $z_j \in \mathcal L(\Sigma )$ of length $L$,
%where $[x]_{n}^m =(x_n x_{n+1} \cdots x_{m})$ for $x =(x_l)_{l\in \mathbb Z}$ and $-\infty \leq n\leq m\leq \infty$ and $[x]_{n}^m = [x]_{0}^m$.
where $[x]_n^{m} =(x_n x_{ n+1} \cdots x_{m})$ for $x =(x_m)_{m\in \mathbb Z}$.
Notice that we have no requirement on $\hat{x}_m$ with $m\leq -1$.
Fix $\epsilon >0$ such that the $\epsilon$-neighborhood of $p^j$  is included in $A_j$ for each $j=0,1$.
%(resp.~$B$).
Let $m_0$ be a positive integer such that
\[
 \sum _{\vert m\vert > m_0 } \frac{1}{\beta ^{\vert m\vert }}
 <\epsilon.
\]
In other words,
\[
x _m =y_m \quad \text{for all $\vert m \vert \leq m_0$} \quad \Rightarrow \quad d_{\Sigma }(x,y) <\epsilon .
\]
Since $(n_j)_{j\in \mathbb N}$ is increasing, one can find $j_1$ such that $n_j \min\{L_1, L_2\}\geq 2m_0+1$ for all $j\geq j_1$.
For each $j\geq j_1$, consider $\hat{x}^j=\sigma ^{N_{j-1} +L+m_0} (\hat{ x} )$.
In the case that $j$ is odd, it satisfies
\[
[\hat{x}^j  ]_{-m_0}^{n_j L_1 - m_0 -1} =[p^1] _{0}^{n_j L_1 -1}.
%[\hat{x}^j  ]_{-m_0}^{m_0} = [p^1] _{0}^{2m_0}, \quad [\hat{x}^j  ]_{m_0+1}^{n_j L_1 - m_0 -1} =[p^1] _{0}^{}
\]
Similar expression with $L_0$, $p^0$ instead of $L_1$, $p^1$  holds if $j$ is even. Thus,
\begin{equation}\label{eq:0211}
d_{\Sigma }\left(\sigma ^{N_{j-1} + L+\ell} (\hat{x}) ,\sigma ^\ell (p^{(j\;\mathrm{mod}\;
2)})\right) < \epsilon
 \quad \text{for all $m_0\leq \ell \leq n_jL_{(j \; \mathrm{mod} \;2 )} -m_0-1$}.
\end{equation}
Take a continuous function $\psi :\Sigma  \to \mathbb R$ such that
\begin{equation}\label{eq:0904e}
 \text{$0\leq \psi  (x)  \leq 1$ for each $x\in \Sigma$,} \quad
  \text{$\psi (x) = j$ for each  $x\in A_j$}
  % with $j=0,1$.}
  % \quad
%  \text{ $\phi (x) =0$ on $A_0$.}
  \end{equation}
  with $j=0, 1$, and define   a  function $\phi : \Sigma _\rho \to \mathbb R$ by
  \[
  \phi (x,z) = \psi (x) + \frac{z}{\rho (x)} (\psi \circ \sigma (x) - \psi (x)).
  \]
  Then, it is straightforward to see that $\phi$ is continuous and
$\psi =\iota (\phi )$.
Since the $\epsilon$-neighborhood of $p^{j}$  is included
in $A_j$,  by \eqref{eq:0211}
\begin{equation}\label{new}
\psi \circ \sigma ^{N_{j-1} + L+\ell} (\hat{ x} )=\begin{cases} 1 &
\text{if $j$ is odd} \\
0 & \text{if $j$ is even}
\end{cases}
\quad \text{for all $m_0\leq \ell \leq n_j L_{(j \; \mathrm{mod} \;2 )} -m_0 -1$}.
\end{equation}
On the other hand, by \eqref{eq:0904e}, it also holds  that $\psi
\circ \sigma ^{\ell } (\hat{x} ) \geq 0$ for all $\ell \geq 0$.
Therefore, taking into account \eqref{eq:0911} and~\eqref{new}, for each odd $j$ we have
%taking $j$ above as $j=2m-1$ we have
\begin{equation*}
\begin{aligned}
 \frac{S_{N_{j}}\psi (\hat{x})}{S_{N_{j}}\rho (\hat{x})} &\geq
 \frac{1}{CN_{j}} \sum_{n=0}^{N_j-1} \psi \circ \sigma ^n(\hat{x})
 >  \frac{ 1 }{CN_j} \sum_{\ell=m_0}^{n_j-m_0-1} \psi \circ
 \sigma ^{N_{j-1}+L+\ell}(\hat{x}) \\ &=  \frac{ (n_{j}L_1 -m_0 -1) -m_0 +1}{CN_{j}}
 =\frac{n_{j}L_1}{CN_{j}} -\frac{ 2m_0 }{CN_{j}}.
 \end{aligned}
 \end{equation*}
Similarly, for even $j$, with a discrete interval $I_j:= \{ n\in \mathbb Z \mid 0\leq n\leq N_j -1 \} \setminus \{  N_{j-1} +L + \ell \mid m_0\leq \ell \leq n_j L_{0} -m_0 -1 \}$,
\begin{equation*}
\begin{aligned}\frac{S_{N_{j}}\psi(\hat{x})}{S_{N_{j}}\rho (\hat{x})}
 &\leq \frac{1}{C^{-1}N_j} \sum _{n=0}^{N_j-1} \psi \circ \sigma ^n (\hat{x})
  = \frac{1}{C^{-1}N_j}  \sum _{n\in I_j} \psi \circ
 \sigma ^{n}(\hat{x})  \\
  &\leq  \frac{ \# I_j}{C^{-1}N_{j}}
  = \frac{ N_{j} - [(n_{j}L_0 -m_0 -1 ) -  m_0 +1]}{C^{-1}N_{j}}
 = C \left(1 -\frac{n_{j}L_0}{N_{j}} \right)+\frac{ 2Cm_0 }{N_{j}}.
 \end{aligned}
 \end{equation*}
Therefore,
by \eqref{eq:0911b} we get that
 \[
\liminf _{m\to\infty } \frac{S_{N_{2m-1} }\psi (\hat{x})}{S_{N_{2m-1}}\rho (\hat{x})} > \frac{2}{3C}> \limsup
_{m\to\infty } \frac{S_{N_{2m} }\psi (\bar x)}{S_{N_{2m}}\rho
(\bar x)}.
\]
In particular, $\hat{x}\in I_\rho ( \psi ,  \sigma )=I_\rho ( \iota (\phi ),  \sigma )$.
This completes the proof of  the first assertion in Lemma \ref{lem:0904}.

%\subsubsection{The item (ii) of Lemma \ref{lem:0904}.}
We next show the items (i) and (ii) of Lemma \ref{lem:0904}.
%  is an immediate consequence of the previous result.
%So, for the  completion of the  proof of Lemma \ref{lem:0904},
 %it suffices to  show the item (ii) of  Lemma \ref{lem:0904}.
%we will prove the former claim.
Let $T$ be a  $C^{1+\alpha}$-diffeomorphism $T$ having a
$u$-conformal elementary horseshoe $\Lambda$ for which $T \vert
_\Lambda $ is topologically conjugate to $\sigma $ by a
homeomorphism $\Pi : \Lambda \to \Sigma$. Then, it is
straightforward to see that
\[
I(\hat \phi , (T\vert _\Lambda )_{\hat \rho }) =\Pi ^{-1}\left( I(
\phi , \sigma _{\rho })\right)
\]
with $\hat \phi $ (induced by the function $\phi$ constructed in the above argument) and $\hat \rho$ given in the statement of Lemma \ref{lem:0904}.
This immediately implies that $I(\hat \phi , (T\vert _\Lambda )_{\hat \rho })  \neq \emptyset$, particularly the item (i), because of the previous result $I( \phi , \sigma _{\rho }) \neq \emptyset$.
Therefore, it follows  from \eqref{eq:0106b} that
\[
\dim _{\mathrm H} I(\hat \phi , (T\vert _\Lambda )_{\hat \rho })   = 1+  \dim _{\mathrm H} I_\rho (\iota (\hat \phi ), T\vert _\Lambda ).
\]
Combining this with \eqref{eq:0219e} and \eqref{eq:0106b}, we get that
\[
\dim _{\mathrm H} I(\hat \phi , (T\vert _\Lambda )_{\hat \rho })
\geq 1+ \mathrm{d}_s(\Lambda) + \mathrm{d}^u (\Lambda )
\]
which completes the proof.
%
%{\color{blue}Sorry for the continuous change of notation. I think
%that the proof is more clear if we ovoid extra notation and proof
%an "abstract" result. Then I want to write here the ideas of the
%proof of the above part in red but whit the notation of the Lemma.
%
%I want to prove that $\dim_H I(\hat\phi,T_{\hat\rho} ) \geq 1 +
%d_s + d^u(\Lambda)$. From~\ref{eq:0106b} we have that $$\dim_H
%I(\hat\phi,T_{\hat\rho} ) =1 + \dim_H I(\iota(\hat\phi),T).
%$$
%Form the first part of the lemma $I(\iota(\hat\phi),T)\not
%\emptyset$ (actually the proof not the statement). Hence from
%remak we have the required claim.
%
%}
\end{proof}

\section{Acknowledgments}
P.G.~Barrientos was supported by Ministerio de Educaci\'on, Cultura
y Deporte (Grant No. MTM2017-87697-P), grant PID2020-113052GB-I00 funded by MCIN %/AEI/10.13039/501100011033
and Conselho Nacional de
Desenvolvimento Cient\'ifico e Tecnol\'ogico (CNPq-Brazil).  Y.~Nakano was partially
supported by JSPS KAKENHI Grant Numbers $19K14575$ and $19K21834$.

\bibliographystyle{alpha2}
\bibliography{biblio-data}

\end{document}